\documentclass[11pt]{article}

\usepackage{amsfonts,amsthm,amssymb,amsmath,mathrsfs}
\usepackage{hyperref,graphicx,color}
\usepackage{pdfsync}
\usepackage[top=1in,bottom=1in,left=1in,right=1in]{geometry}
\usepackage{xcolor}
\usepackage{setspace}

\usepackage[round]{natbib}

\usepackage{authblk}

\usepackage{multirow}
\usepackage{amsmath,amssymb}
\usepackage{tikz}
\usepackage{graphicx}
\usepackage{pgfplots}
\usepackage{caption}
\usetikzlibrary{backgrounds}
\usetikzlibrary{patterns}
\usetikzlibrary{positioning,calc,shapes,arrows,snakes}
\usetikzlibrary{shapes.arrows,intersections,decorations.pathmorphing,decorations.pathreplacing,decorations.shapes}

\numberwithin{equation}{section}
\numberwithin{table}{section}
\numberwithin{figure}{section}
\newtheorem{definition}{Definition}[section]
\newtheorem{assumption}{Assumption}[section]
\newtheorem{lemma}{Lemma}[section]
\newtheorem{proposition}{Proposition}[section]
\newtheorem{theorem}{Theorem}[section]
\newtheorem{corollary}{Corollary}[section]

\newcommand{\R}{\mathbb{R}} 

\newcommand{\e}{\mathbf{e}} 
\newcommand{\fl}[1]{\lfloor{#1}\rfloor} 
\newcommand{\cl}[1]{\lceil{#1}\rceil} 
\newcommand{\id}[1]{{\mathbf 1}_{\{{#1}\}}} 
\newcommand{\dist}{\rho} 

\newcommand{\up}[1]{{#1}^{\uparrow}} 
\newcommand{\down}[1]{{#1}^{\downarrow}} 

\newcommand{\pr}{\mathbb{P}} 
\newcommand{\E}{\mathbb{E}} 
\newcommand{\prob}{\mathbb{P}} 

\newcommand{\0}{\mathbf{0}} 

\newcommand{\spt}{\tau} 
\newcommand{\ts}{\sigma} 
\newcommand{\arp}{\mathcal{X}} 
\newcommand{\ja}{\xi} 
\newcommand{\lt}{\ell} 
\newcommand{\ulb}{\up w_\0(\lt)} 
\newcommand{\uub}{\down w_\0(\lt)} 

\usepackage[normalem]{ulem}

\title{Instantaneous Control of Brownian Motion \\
with a Positive Lead Time}

\author{Zhen Xu, Jiheng Zhang, and Rachel Q. Zhang \\
The Hong Kong University of Science and Technology \\
}

\begin{document}

\maketitle

\begin{abstract}
Consider a storage system where the content is driven by a Brownian motion absent  control.
At any time, one may increase or decrease the content at a cost proportional to the amount of adjustment.
A decrease of the content takes effect immediately, while an increase is realized after a fixed lead time $\lt$.
Holding costs are incurred continuously over time and are a convex function of the content. 
The objective is to find a control policy that minimizes the expected present value of the total costs.
Due to the positive lead time for upward adjustments, one needs to keep track of all the outstanding upward adjustments as well as the actual content at time $t$ as there may also be downward adjustments during $[t,t+\lt)$, i.e., the state of the system is a function on $[0,\ell]$.  
To the best of our knowledge, this is the first paper to study instantaneous control of  stochastic systems in such a functional setting.  
We first extend the concept of $L^\natural$-convexity to function spaces and establish the $L^\natural$-convexity of the optimal cost function. 
We then derive various properties of the cost function and identify the structure of the optimal policy as a state-dependent two-sided reflection mapping making the minimum amount of adjustment necessary to keep the system states within a certain region.  
\end{abstract}

\section{Introduction}
\label{sec:introduction}

Consider a storage system, such as an inventory or cash fund, whose content fluctuates according to a Brownian motion absent control. 
A convex holding cost of the content is incurred continuously.
At any time, a controller may initiate an upward adjustment to increase the content, which is realized after a lead time, and/or a downward adjustment to decrease the content, which takes effect immediately. 
Both upward and downward adjustments incur a variable cost. 
The objective is to find a control policy that minimizes the expected discounted cost over an infinite planning horizon. 
 
Absent the lead time, the state of the problem is one dimensional, and \cite{HarrisonTaksar1978,HarrisonTaksar1983} show that an optimal control policy can be characterized by two closed-form control limits. 
The method used to analyze the problem is referred to as a lower-bound approach by \cite{DaiYao2013a} and involves three steps. 
(1) Based on the optimality equations, heuristically derive some differential inequalities of the optimal cost function, with at least one equation being tight. 
This is know as the Hamilton-Jacobi-Bellman (HJB) equation.
(2) For a control limit policy, first obtain a set of ordinary differential equations (ODEs) of the cost function and then solve those equations. 
(3) Find the control limits under which the cost function is continuously differentiable and hence optimal.

The problem becomes much more complicated, however, when there is a positive lead time $\lt$ for upward adjustments. 
This is because the on-hand inventory at $t+\lt$ cannot be predicted solely from the inventory position  at any time $t$  as there may be downward adjustments in $[t,t+\lt)$. 
One needs to keep track of the amount and timing of each outstanding upward adjustment as well as the content on-hand at any time, or the state of the system is a function on $[0,\lt]$.
Thus, step (2) of the lower bound approach  will only result in partial differential equations (PDEs) with an uncountable number of unknown boundary conditions, which are almost impossible to solve.

To derive and prove the structure of the optimal control policy in the presence of a positive lead time, we follow step (1) to heuristically derive an HJB equation based on two optimality conditions, optimizing the timing and amounts of adjustments, respectively. 
The similarity between our analysis and the lower bound approach in \cite{HarrisonTaksar1983} stops here and we proceed with the following steps, each of which involving challenging and deep mathematical analysis.  
(2) Extend the concept of  $L^\natural$-convexity defined on finite dimensional spaces and introduced by \cite{Murota2005} to a function space, and show that the optimal cost function is the limit of the costs of a series of periodic review systems and hence is $L^\natural$-convex in our state space.
This is one of the key steps in our analysis and a fundamental building block.
(3) Derive some properties of the optimal cost function using the $L^\natural$-convexity of the cost function, and identify two sets of states in which an upward and a downward adjustment are needed, respectively. 
These two sets also reveal the boundaries of the PDEs for the HJB equation.
(4) Construct a state dependent two-sided reflection policy making the minimum amount of upward or downward adjustment necessary to prevent the state from entering into the two sets and prove it is optimal. 
Such a policy is much more complicated than that in \cite{HarrisonTaksar1983} and the proof of its optimality  requires the establishment of properties such as the monotonicity, Lipschtiz continuity, and complementarity of the policy.

To the best of our knowledge, this is the first paper to consider instantaneous control of stochastic systems where the state is a function on a continuous time interval.
Existing methods can only deal with systems with single dimensional states, e.g., zero lead time for both upward and downward adjustments in our problem. 
For periodic control problems, except for those with states of one or two dimensions, the common approach is to establish the $L^\natural$-convexity of the optimal cost function, with which a threshold policy can be easily shown to be optimal.  
Such an approach cannot be applied directly to problems with instantaneous control as $L^\natural$-convexity is only defined on finite dimensional spaces. 
As one can see, identifying the optimal policy is nontrivial even after extending and applying the concept of $L^\natural$-convexity to a function space (i.e., step (2)), and requires additional challenging steps, i.e, steps (3) - (4) mentioned above. 
     
The remainder of this paper is organized as follows.
In $\S 2$, we provide a brief summary of relevant literature. 
In $\S 3$, we present a precise mathematical formulation of the Brownian control problem. 
We then derive two optimality conditions and provide a heuristic derivation of an HJB equation.  
In $\S 4$, we extend the concept of $L^\natural$-convexity to a function space, and show that the optimal cost function is the limit of the costs of a series of periodic review systems and hence is $L^\natural$-convex.
In $\S 5$, we provide various properties of the optimal cost function, which lead to the optimal control being a state-dependent two-sided reflection policy in $\S 6$. 
We discuss the general case with positive lead times for both upward and downward adjustments in $\S 7$. 

\section{Literature Review}
Research on the stochastic control of Brownian motion dates back to \cite{Bather1966} and the early work was aimed at minimizing the total expected discounted costs. 
\cite{ConstantinidesRichard1978} show that a control band policy is optimal when there is a fixed cost for upward and downward adjustments and \cite{HST1983} develop a method to find the optimal bands. 
\cite{Davis1993} and \cite{OksendalSulem2009} show the equivalence of this control problem to a sequence of optimal stopping problems. 
All of these papers assume that the holding cost is linear. \cite{DaiYao2013b} extend this work to a general convex holding cost function. 
\cite{HarrisonTaksar1978,HarrisonTaksar1983} prove that a control limit policy is optimal absent fixed costs under linear and convex holding costs, respectively, and the latter also provides a procedure for computing the optimal limits. 
The methodology used in these papers is the three-step approach described in the Introduction.  
Later, these policies are shown to be optimal also under the average cost criteria by  \cite{ODV2008} and \cite{DaiYao2013a} with fixed costs when the holding cost is linear and convex, respectively,  and by \cite{Taksar1985} without a fixed cost.
 
Note that all of the abovementioned work assumes away a positive lead time for upward or downward adjustments, except \cite{OksendalSulem2009} which show that, with some additional assumptions which will be discussed in Section~\ref{sec:discussion}, the problem where the lead times for upward and downward adjustments are the same can be reduced to one with  zero lead times.      
 
Since the state in our problem is on a function space, the literature on  $L^\natural$-convexity which extends convexity to multiple dimensions is also relevant. 
We refer to \cite{Zipkin2008} for an excellent summary of the development of the concept and its application in inventory management.
By establishing the  $L^\natural$-convexity of the optimal cost function, \cite{Zipkin2008}  develops a new approach to the structural analysis of the standard, single-item, lost-sales inventory system with a linear ordering cost and a positive replenishment lead time. 
This concept is also used in the structural analysis of problems where the state is of a finite dimension, e.g., inventory-pricing control with lead times (\cite{PCF2012}) and perishable inventory systems (\cite{CPP2014}). 
In our paper, we will extend $L^\natural$-convexity to a function space.  

The two-sided reflection policy shown to be optimal for our problem is inspired by  \cite{Skorokhod1961} and \cite{Skorokhod1962} which solve the stochastic differential equation for a reflecting Brownian motion.
The idea of the reflection mapping is widely used in the study of queueing systems. 
For example, \cite{HarrisonReiman1981} and \cite{Reiman1984} obtain the heavy-traffic limits for some open queueing network using multidimensional reflection mappings. 
We refer to \cite{ChenYao2001} and \cite{Whitt2002} for more in-depth knowledge about reflection mappings.

\section{Model Description}
In this section, we formulate the problem mathematically and heuristically derive the Hamilton-Jacobi-Bellman(HJB) equation. 

\subsection{Problem Formulation}
\label{sec:model-formulation}
 \subsubsection{Modeling Details}
Let $\Omega$ be the set of all continuous functions $\omega:[0,\infty) \rightarrow \R$, and $W_t:\Omega \rightarrow \R$ be the coordinate projection map $W_t(\omega) = \omega(t)$ for $t\geq0$. 
Denote by $\mathscr{F} =\sigma (W_t,t\geq 0)$ the smallest $\sigma$-field such that  $W_t$ is $\mathscr{F}$-measurable and $\mathscr{F}_t= \sigma(W_s, 0\leq s\leq t)$ for each $t\geq0$. 
Also let $\pr$ be the unique probability measure on $(\Omega,\mathscr{F})$ such that $W_t$ is a Brownian motion with drift $\mu$ and  variance $\sigma^2$, and $\E$ be the associated expectation operator.

Now consider a storage system, such as an inventory or bank account, whose content $H_t$, $t\geq 0$, fluctuates according to a Brownian motion $W_t$ with drift $\mu$ and variance $\sigma^2$, without any control.
Holding costs are incurred continuously at the rate $h(H_t)$ where $h$ is convex with $h(0)=0$. At any time, we may take an action to cause the storage level to jump by a positive amount after a fixed lead time $\lt$ or by a negative amount which takes effect immediately. 
An upward adjustment incurs a variable cost $\up k$, while a downward adjustment incurs a variable cost $\down k$. 
Thus, the cost for an upward $\up \xi$ and/or downward $\down \xi$ adjustment at any given time is given by 
\begin{equation}
  \label{eq:control-cost}
  \phi(\up \ja,\down \xi)= \up k \up \xi + \down k \down \xi.
\end{equation}

When $\lt=0$, the problem reduces to that in \cite{HarrisonTaksar1983}. 
With a positive lead time for upward adjustments, the problem becomes much more complicated for the following reasons. 
(i) As instantaneous downward adjustments can occur at any time, by itself the {\it inventory position} at any time $t$ cannot predict the content on-hand and hence the expected holding cost at time $t+\lt$.  
One needs to keep track of all the upward adjustments that will be realized in $[t,t+\ell)$, or a profile of outstanding upward adjustments. 
(ii) With continuous time, such a  profile is a function on $[0,\ell]$. 
Dynamic control with infinite dimensional state variables is well known to be extremely challenging and there has been little work in the literature. Next, we define the state and decision variables, and provide the system dynamics of the problem.   

\begin{enumerate}
\item The state variables: Let $\arp_t(u) \in \R$ be the content of the system plus the total amount of outstanding upward adjustments at time $t$ that will be realized by $t+u$. 
Then, $\arp_t(0)$ is simply the content of the system at time $t$. 
For technical purposes, we include $\arp_t(u)=\arp_t(\lt)$ for $u>\lt$ in our state. 
Thus, $\arp_t(u), u \geq 0$, is right-continuous, non-decreasing and constant for $u \geq \lt$.  

Let $\arp_t=\{\arp_t(u),u\ge 0\}$ be the state of the system at time $t$ and $\mathbb{D}$  be the set of all possible states. 
That is, $\mathbb{D}$ is the set of all functions on $\R_+$ with the following properties: (1) right-continuous
on $[0,\infty)$ with  left limits in $(0,\infty)$, and (2) non-decreasing. 
For convenience, we denote $\mathcal I=\{ \mathcal I(u)=1, u\geq 0\} \in \mathbb D$ and $\arp+a=\{\arp(u)+a,u \geq 0\} \in \mathbb D$ for $a \in \R$.  

\item The decision variables: Let $\up Y(t)$ and $\down Y(t)$ be stochastic processes adapted to the filtration $\mathscr{F}_t$ for all $t \ge 0$, representing the cumulative upward and downward adjustments up to time $t$, respectively.  
Thus, $\up Y$ and $\down Y$ are non-decreasing functions.
For convenience, let $\pi=(\up{Y},\down{Y})=\{(\up Y(t),\down Y(t)): t \geq 0\}$ represent a control policy over the planning horizon such  that any control at time $t$
is based on information that has been revealed up to $t$. 
 
\item The system dynamics: For $t>0$,
\begin{equation}
  \label{equ:dynamics}
  \arp_t(u) =\left\{ \begin{array} {ll}
  \arp_{0}(u+t) + W_{t}+\up{Y}(t+u-\lt)-\down{Y}(t), & u \le \lt,\\
  \arp_t(\lt), & u> \lt.
  \end{array}\right.
\end{equation} 
That is, apart from $W_t$, $\arp_t(u)$ includes the content at time $0$ plus the upward adjustments made before time $t+u-\lt$ if $u \leq \lt$ or $t$ otherwise, minus the downward adjustments made up to $t$.  
When $u \leq \ell$, $\arp_{0}(u+t)$ is the content of the system at time $0$ plus the upward adjustments made before time $0$ that will be realized by $t+u$. 
Among the upward adjustments made during $[0,t)$, only those made before $t+u-\ell$ will be realized by time $t+u$, which is $\up Y(t+u-\lt)$.
Thus, the content on hand at $t$ can be written as $H_t=\arp_{0}(t) + W_{t}+\up{Y}(t-\lt)-\down{Y}(t)$.
\end{enumerate}

\subsubsection{The Cost Function}
For any given policy $\pi$ and initial state $\arp\in\mathbb{D}$, the total expected cost can be written as 
\begin{equation}
  \label{eq:trans_cost_function}
  C(\arp,\pi)
  =\E\left[
    \int_0^{\infty}e^{-\gamma t}h(\arp_t(0)) dt 
     + \int_0^{\infty}e^{-\gamma t}
        (\up k d \up Y(t)+\down k d \down Y(t))
  \right],
\end{equation}
where $\gamma$ is the discount rate. 
We impose the following mild assumptions on the holding cost function for the rest of this paper. 

\begin{assumption}
  \label{assum:orig-h}
The holding cost function $h:\R \to \R^+$ satisfies the following conditions: $(1)$ $h(\cdot)$ is convex and piece-wise $C^2$-continuous; 
$(2)$ $h(0)=0$; and
$(3)$ there exists $M>0$ such that $|h'(\cdot)|\leq M$.
\end{assumption}

Parts (1) and (2) of Assumption \ref{assum:orig-h} guarantee that it is never optimal to make a downward adjustment exceeding the available content at any time. 
Without loss of generality, we will only consider feasible policies that result in finite control costs, i.e., 
\begin{equation}
\label{equ:feasiblecondition}
  \E\left[\int_0^{\infty}e^{-\gamma t} (d \up Y(t)+ d \down Y(t))\right] < \infty.
\end{equation}
Thus, under Assumption~\ref{assum:orig-h}, a policy $\pi$ is feasible if and only if $C(\arp,\pi)$ is finite.
Denote by $\Pi$  the set of all such control policies and by $C^*(\arp)=\inf\limits_{\pi\in\Pi}\{C(\arp,\pi)\}$  the optimal cost.

The following proposition shows that the optimal cost $C^*(\arp)$ is Lipschitz
continuous on $\mathbb{D}$. 
All the proofs in the paper are either in the main body or can be found in the Appendix. Since the states are functions, we define the distance between two states $\arp $ and $\arp'\in\mathbb{D}$ as $d(\arp,\arp')=\int_0^{\infty} e^{-\gamma t}|\arp(t)-\arp'(t)|dt$. 
It is easy to see that the space $\mathbb{D}$ is a complete metric space under the distance $d(\cdot,\cdot)$. 
\begin{proposition}
  \label{prop:Lips-con of C^*}
Under Assumption~\ref{assum:orig-h}, $C^*(\arp)$ is Lipschitz continuous.
That is, for any states $\arp$ and $\arp'$, $|C^*(\arp)-C^*(\arp')| \leq Md(\arp,\arp')$.
\end{proposition}

\subsection{Heuristic Derivation of the Hamilton-Jacobi-Bellman(HJB) Equation}
\label{sec:Heuristic derivation of the optimality equation}

We first note that, for any given initial state $\arp \in \mathbb D$, the optimal cost should satisfy the following optimality conditions: 
\begin{eqnarray}
C^*(\arp) &=& \inf\limits_{\up \xi \ge 0, \;\down\xi \ge 0}
    \left\{
    \phi(\up\xi, \down\xi)+C^*(\Phi_{\up\xi,\down\xi}(\arp))
    \right\}, \label{optimality1} \\
C^*(\arp) &=&  \inf_{s\geq0} \left\{\E
    \left[
    \int_0^{s}e^{-\gamma u}h(\arp(u)+W_u)du+e^{-\gamma s}C^*(\ts_s(\arp)+W_s)
\right] \right\}, \label{optimality2}
    \end{eqnarray}
where $s$ is a stopping time and
\begin{eqnarray}
  \Phi_{\up \ja,\down \xi}(\arp) 
  &=&\left\{\arp(u)-\down \xi+ \up \xi {\bf 1}_{\{u \geq \ell\}}:u\geq 0
\right\}, \label{Phi} \\
  \ts_s(\arp)  &=& \{\arp(s+u),u\geq 0\} \nonumber
\end{eqnarray}
are the states after  an adjustment $(\up \xi,\down \xi)$ is made and after
a period
of time $s$ with no adjustment for a given initial state $\arp$, respectively.
Let
\begin{equation}
  \label{eq:cost-with-control}
  C(\arp,\up \xi,\down \xi)=\phi(\up\xi, \down\xi)+C^*(\Phi_{\up\xi,\down\xi}(\arp)) 
\end{equation}
be the minimum cost under a given adjustment $(\up\xi,\down \xi)$.
Assume for now that 
$\frac{\partial C(\arp,\up \xi,\down \xi)}{\partial \up \xi}$ and  $\frac{\partial C(\arp,\up \xi,\down \xi)}{\partial \down \xi}$ exist, which we will prove later. 
Then, with a small amount of adjustment $\epsilon$,  
\begin{eqnarray}
  \label{eq:heuristic-1}
  C(\arp,\epsilon,0)&=&C^*(\arp)+\frac{\partial C(\arp,0,0)}{\partial \up\ja}
\epsilon+o(\epsilon),\\ 
  \label{eq:heuristic-2}
  C(\arp,0,\epsilon) &=& C^*(\arp)+\frac{\partial C(\arp,0,0)}{\partial \down\ja}
\epsilon+o(\epsilon). 
\end{eqnarray}
If $(\up\xi,\down \xi)=(0,0)$, i.e., no adjustment is made at time $0$, absent  further adjustment, the state at time $s>0$ becomes $\ts_s(\arp)+w$ for any realization of $W_s=w$. We define
\begin{equation}
  \label{eq:cost-no-control}
  V_{\arp}(w,s)=C^{*}\left(\ts_s(\arp)+w\right).
\end{equation} 
If no adjustment is made for $\epsilon$ amount of time, then, by Ito's formula, the minimum expected discounted cost becomes
\begin{align}
\E\left[\int_0^{\epsilon} e^{- \gamma t} h(\arp_t(0)) dt + e^{- \gamma \epsilon} V_{\arp}(W_{\epsilon},\epsilon)\right]
= C^*(\arp)
   +[\Gamma V_{\arp}(0,0)-\gamma V_{\arp}(0,0)+h(\arp(0))]\epsilon+o(\epsilon) \label{eq:heuristic-3}
\end{align}
where the operator $\Gamma=\frac{\partial}{\partial s}+\frac{\sigma^2}{2}\frac{\partial^2}{\partial w^2}+\mu\frac{\partial}{\partial w}$.
Thus, for any given $\arp$,
\begin{equation}
  \label{eq:optimality-heuristic}
  \big[
  \Gamma V_{\arp}(0,0)-\gamma V_{\arp}(0,0)+h(\arp(0))
  \big]
  \vee \frac{\partial C(\arp,0,0)}{\partial \up\ja}
  \vee \frac{\partial C(\arp,0,0)}{\partial \down\ja}=0
\end{equation}
which is precisely the HJB equation.
This equation involves a PDE with an uncountable number of unknown boundary conditions and no known method is available to solve it directly. 
Instead, we will solve the problem by first establishing the $L^\natural$-convexity of the optimal cost function on function spaces.
And we will give a solution to the HJB equation in Theorem~\ref{thm:HJB-equation}.
\section{The $L^\natural$-convexity of the Optimal Cost Function} 
\label{sec:L-convex}

Since the concept of $L^\natural$-convexity is defined on $\R^n$, we first study a periodic version of the  problem. 
We then extend the concept of $L^\natural$-convexity from $\R^n$ to $\mathbb D$ by linking the problem to the limit of a series of periodic problems.

\subsection{A Periodic Review System}
\label{sec:periodic}
Consider a periodic review of the system with period length  $\frac{\ell}{n}$,  i.e., an upward adjustment takes $n$ periods. 
In such a system, the state in any period is an $n$-dimensional vector denoted by 
$\mathbf{x}_t=(x_{t,0},x_{t,1},\ldots, x_{t,n-1})$ where $x_{t,0}$ is the current content of the system and $x_{t,i}$, $1 \leq i \leq n-1$, is the content of the system plus the total outstanding upward movement that will be realized from period $t+1$ to $t+i$. 
Letting $\up y_t$ and $\down y_t$ be the upward and downward adjustments in period $t$, we obtain the following dynamics:
\begin{equation}
  \label{eq:dynamics-periodic}
  \mathbf{x}_{t+1} = (x_{t,1},x_{t,2},\cdots,x_{t,n-1},x_{t,n-1}+\up y_t)-\down y _t\e+w_t \e
\end{equation}
where $\e$ is a vector of all 1's whose dimension will be clear from the context and $w_t=W_{\frac{(t+1)\lt}{n}}-W_{\frac{t\lt}{n}}$ is the random change caused by the Brownian motion.
Let $\mathcal{N}_a$ represent a normally distributed random variable with mean $a \mu$ and variance $a \sigma^2$ for any $a>0$. 
Then, the discount rate becomes $\alpha=e^{-\gamma \frac{\lt}{n}}$ and holding cost is given by $h^n(x)=\E\left[\int_0^{\frac{\lt}{n}} e^{- \gamma s} h\left(x+\mathcal{N}_{\frac{\lt}{n}}\right)ds\right]$ in the periodic system.

Next, we present definitions where the concept of $L^\natural$-convexity can be found in \cite{Zipkin2008}, and show that the optimal cost function for the periodic system is $L^\natural$-convex.
\begin{definition}
Let $f$ be a function on $\R^n$. 
\begin{enumerate}
\item\ $f$ is submodular if for any $\mathbf{x}_1,\mathbf{x}_2\in\R^n$,
$f(\mathbf{x}_1)+f(\mathbf{x}_2)\geq f(\mathbf{x}_1\vee \mathbf{x}_2)+f(\mathbf{x}_1\wedge\mathbf{x}_2)$.
\item $f$ is $L^\natural$-convex if the function $g(\mathbf{x},\xi)=f(\mathbf{x}-\xi\e)$ is submodular in $\R^{n+1}$.
\end{enumerate}
\end{definition}
Thus, a function $f$ is $L^\natural$-convex if and only if, for any $\mathbf{x}_1,\mathbf{x}_2\in\R^n$ and  $\xi_1,\xi_2 \in\R,$
\[  f(\mathbf{x}_1-\xi_1 \e)+f(\mathbf{x}_2-\xi_2 \e)\geq
    f(\mathbf{x}_1\vee \mathbf{x}_2-(\xi_1 \vee \xi_2)\e)
    +f(\mathbf{x}_1 \wedge \mathbf{x}_2-(\xi_1 \wedge \xi_2)\e).
\]
To show the $L^\natural$-convexity of the optimal cost function for the periodic system, we define $C^{T,n}_t(\mathbf{x}_t)$ as the optimal cost function from period $t$ to $T$ for a given $(T,n)$ and state $\mathbf{x}_t$. 
Then,
\begin{eqnarray}
C^{T,n}_t(\mathbf{x}_t)=\min_{\up y_t, \down y_t\geq0}\left\{c^{T,n}_t(\mathbf{x}_t,\up
y_t,\down y_t) \right\} \nonumber,
\end{eqnarray}
where
\begin{equation*}
 c^{T,n}_t(\mathbf{x}_t,\up y_t,\down y_t)=\up k \up y_t +\down k \down y_t
+\alpha \E \left[C^{T,n}_{t+1}(\mathbf{x}_{t+1})+h^n(x_{t,0}- \down y_t)\right]
\end{equation*}
for $0 \leq t \leq T-1$ and $C^{T,n}_T(\mathbf{x}_T)=0$.

\begin{proposition}
\label{prop:midstep-convex}
$c^{T,n}_t(\mathbf{x}_,\up y_,\down y)$ is $L^\natural$-convex in 
$(\mathbf{x}_,x_{n},\down y)$ and $C^{T,n}_t(\mathbf{x})$ is $L^\natural$-convex in $\mathbf{x}$.
\end{proposition}

By Theorem~6.2.3 of \cite{Puterman1994} $
  C^{\infty,n}(\mathbf{x}) := \lim\limits_{T\rightarrow \infty}\left\{C^{T,n}_0(\mathbf{x})\right\}<\infty$ is the unique solution to the optimality equation $C^{\infty,n}(\mathbf{x})=\min\limits_{\up y, \down y\geq0}\left\{c^{\infty,n}(\mathbf{x},\up y,\down y) \right\}$ where
\begin{equation*}
c^{\infty,n}(\mathbf{x},\up y,\down y)=\up k \up y +\down k \down y
+\alpha \E\left[
C^{\infty,n}((x_{1},x_{2},\cdots,x_{n-1},x_{n-1}+\up y)-\down y\e+w_t \e)+h^n(x_{0}-\down y)
\right]
  \end{equation*}
and hence we have the following theorem.
\begin{theorem}
  \label{theo:c is L-convex}
$C^{\infty,n}(\mathbf{x})$ is $L^\natural$-convex and hence the optimal cost for the infinite horizon periodic review system for any given $n$.
\end{theorem}
Thus, there exists a unique optimal adjustment $(\up y,\down y)$ for any given $\mathbf x$ and the optimal $\up y (\down y)$ is increasing (decreasing) in $\mathbf x$, where the order of $\mathbf x$ in $\R^n$ is defined in the usual way of componentwise comparison.

\subsection{The Continuous Review System}
\label{sec:L-natural-convexity}

Since the state $\arp$ is defined on $\mathbb{D}$ rather than $\R^n$, we need to extend the concept of $L^\natural$-convexity to $\mathbb{D}$.
The $L^\natural$-convexity of $C^*(\arp)$ will enable us to construct an optimal policy in Section~\ref{sec:optimal-policy-zero}.
\begin{definition}
Suppose that $\arp_1,\arp_2 \in \mathbb{D}$.  
\begin{itemize}
\item Order: $\arp_1 \succeq \arp_2$ if $\arp_1(u) \ge \arp_2(u)$ for any $u \ge 0$,
and $\arp_1 \preceq \arp_2$ if $\arp_1(u) \le \arp_2(u)$ for any $u \ge 0$.
\item Max and Min Operations: $\arp_1 \vee \arp_2 =\{\arp_1(u) \vee \arp_2(u), u \ge 0\}$ and $\arp_1 \wedge \arp_2 =\{\arp_1(u) \wedge \arp_2(u), u \ge 0\}$.
\end{itemize}
\end{definition}
\begin{definition}
\label{def:L-natural-convexity}
A function $F$ on  $\mathbb{D}$ is $L^\natural$-convex if, for any $\arp_1,\arp_2\in \mathbb{D}$ and $\xi_1,\xi_2\in\R$,
\begin{equation*}
  F(\arp_1-\xi_1)+F(\arp_2-\xi_2)
  \geq
  F({\arp_1}\vee{\arp_2}-({\xi_1}\vee{\xi_2}))
  +F({\arp_1}\wedge{\arp_2}-({\xi_1}\wedge{\xi_2})).  
\end{equation*}
\end{definition}

To connect the periodic review systems with our original one, for any given state $\arp$ and policy $\pi$, consider the following discretized state $\arp^n$ and policy $\pi^n$ which makes adjustments only at multiples of $\frac{\lt}{n}$. 
It is easy to see that $\arp^n$ and $\pi^n$ approach  $\arp$ point-wise and $\pi$, respectively, as $n \to \infty$.  
\begin{enumerate}
\item 
The state $\arp^n$ is such that 
\begin{equation}
  \label{equ:continuous-to-discrete}
  \mathcal{X}^n(u) = 
  \left\{ 
    \begin{array} {lll}
      \mathcal{X}(\frac{\lt}{n}), & \mbox{if } 0 \le  u \le \frac{\lt}{n}, \\
      \mathcal{X}(\frac{i\lt}{n}), & \mbox{if }\frac{(i-1)\lt}{n} < u \le \frac{i\lt}{n},\ 
                                    i=2,3,\cdots,n,\\
      \arp(\lt), & \mbox{if }u > \lt.
    \end{array}\right.
\end{equation} 
Let \begin{equation*}
  \mathbf{x}^n
  =\left(
    \mathcal{X}\left(\frac{\lt}{n}\right),
    \mathcal{X}\left(\frac{2\lt}{n}\right),
    \cdots,
    \mathcal{X}\left(\frac{(n-1)\lt}{n}\right),
    \arp(\lt)
  \right).
\end{equation*}

\item The policy $\pi^n=(Y^{n\uparrow},Y^{n\downarrow})$ is such that
\begin{eqnarray}
\label{eqn:periodic-policy}
  Y^{n\uparrow}(t) = \sum_{i=0}^{\fl{\frac{nt}{\lt}}} \xi_i^{n\uparrow} 
  \mbox{ and } 
  Y^{n\downarrow}(t) = \sum_{i=0}^{\fl{\frac{nt}{\lt}}} \xi_i^{n\downarrow}
\end{eqnarray}
where $(\xi_0^{n\uparrow},\xi_0^{n\downarrow})=(\up Y(0),\down Y(0))$ and $(\xi_i^{n\uparrow},\xi_i^{n\downarrow})=\left(\up Y(\frac{i \lt}{n})-\up Y\left(\frac{(i-1) \lt}{n}\right),\down Y\left(\frac{i \lt}{n}\right)-\down Y\left(\frac{(i-1) \lt}{n}\right)\right)$ for $i=1,2,\cdots$.

\end{enumerate}
Then, the cost of the system for a given $(\arp^n,\pi^n)$ is given by 
\begin{equation}
\label{equ:trans_dis_cost_finction}
  C(\arp^n,\pi^{n})
  =\E\left[
    \int_0^{\infty}e^{-\gamma t} h(\arp^{n}_t(0)) dt 
  + \int_0^{\infty}e^{-\gamma t} (\up k dY^{n\uparrow}(t)+\down k dY^{n\downarrow}(t))
   \right]
\end{equation} 
where $\arp^{n}_t(\cdot)$ is the corresponding state at time $t$ under $\pi^n$ with the initial state $\arp^n$.
By  \eqref{equ:dynamics}, we also have $\arp^{n}_t(0) \to \arp_t(0)$ as $n \to \infty$ for any $t\ge 0$.
It then follows by \eqref{eq:trans_cost_function}, \eqref{equ:trans_dis_cost_finction} and the Lebesgue's dominated convergence theorem that
\begin{equation}
  \label{eq:lebesgue-conv}
  \lim\limits_{n\rightarrow \infty}C(\arp^n,\pi^n) = C(\arp,\pi).  
\end{equation}
It remains to be shown that the optimal cost of the original problem is the limit of the costs of periodic review systems and hence  is $L^\natural$-convexity by Theorem~\ref{theo:c is L-convex}.
\begin{proposition}  
\label{prop:C^n converge to C}
   $C^*(\arp)=\lim\limits_{n\rightarrow+\infty}C^{\infty,n}(\mathbf{x}^n).$
 \end{proposition}

\begin{proof}
Since $C^*(\arp)$ is the optimal cost, for any $\epsilon >0$, we can find a policy $\pi$ such that $C(\arp,\pi) < C^*(\arp) + \epsilon$.
On the other hand, as $C^{\infty,n}(\cdot)$ is the optimal cost of the periodic review system, $C^{\infty,n}(\mathbf{x}^n) \le C(\arp^n,\pi^n)$.
Then, we have
\begin{equation}
  \label{equ:c < C+}
\limsup\limits_{n\rightarrow+\infty}C^{\infty,n}(\mathbf{x}^n) 
\le \lim\limits_{n\rightarrow+\infty}C(\arp^n,\pi^n) 
=   C(\arp,\pi) 
<   C^*(\arp) +\epsilon. \nonumber
\end{equation}
As $\epsilon>0$ is arbitrary, we have $\limsup\limits_{n\rightarrow+\infty}C^{\infty,n}(\mathbf{x}^n) \le C^*(\arp)$.
Combined with the fact that \\ $\liminf\limits_{n\rightarrow+\infty}C^{\infty,n}(\mathbf{x}^n) \ge \lim\limits_{n\rightarrow+\infty}C(\arp^n) \ge C^*(\arp)$, we have the result.
\end{proof}

\begin{theorem}
  \label{thm:L-natural}
  The optimal cost $C^*(\arp)$ is $L^\natural$-convex in $\mathbb{D}$.
\end{theorem}
\begin{proof}
For any $\arp_1,\arp_2 \in \mathbb{D}$ and their respective $\mathbf{x}^n_1$ and $\mathbf {x}^n_2$, it is clear that $\mathbf{x}^n_1 \vee \mathbf{x}^n_2$ is the vector form of $({\arp_1}\vee{\arp_2})^n=(\arp_1)^n\vee (\arp_2)^n$ and $\mathbf{x}^n_1 \wedge \mathbf{x}^n_2$ is the vector form of $({\arp_1}\wedge{\arp_2})^n=(\arp_1)^n\wedge(\arp_2)^n$. 
For any $\xi_1,\xi_2 \in \R$, by the $L^\natural$-convexity of $C^{\infty,n}(\mathbf{x})$ in Theorem~\ref{theo:c is L-convex},
\begin{equation*}
\label{equ:c^n L-covex}
  C^{\infty,n}(\mathbf{x}_1^n-\xi_1\e)+C^{\infty,n}(\mathbf{x}_2^n-\xi_2\e)\geq
   C^{\infty,n}(\mathbf{x}^n_1 \vee \mathbf{x}^n_2-({\xi_1}\vee{\xi_2})\e)
    +C^{\infty,n}(\mathbf{x}^n_1 \wedge \mathbf{x}^n_2-({\xi_1}\wedge{\xi_2})\e).
\end{equation*}
Letting $n \to \infty$, we see that $C^*(\arp)$ satisfies Definition~\ref{def:L-natural-convexity}.
\end{proof}

\section{Properties of the Optimal Cost Function $C^*(\arp)$}
\label{sec:prop-optimal-function}

\subsection{Impact of Adjustments on the Cost Function}
\label{sec:impact-controls-cost}

Recall the function $C(\arp,\up \xi,\down \xi)$ and their partial derivatives $\frac{\partial C(\arp,\up \xi,\down \xi)}{\partial \up \xi}$ and $\frac{\partial C(\arp,\up \xi,\down \xi)}{\partial \down \xi}$ introduced in Section~\ref{sec:Heuristic derivation of the optimality equation}.
A quick fact is that the $L^\natural$-convexity of $C^*(\arp)$ immediately implies that the cost function $C(\arp,\up \xi,\down \xi)$ is convex and differentiable in $\up \xi$ and $\down \xi$.
The following properties of the partial derivatives will help identify the control regions and consequently construct the optimal policy in Section~\ref{sec:optimal-policy-zero}.

\begin{lemma}
Monotonicity of the derivatives:
\label{prop:mono-boundary-con}
\begin{enumerate}
\item If $\arp_1 \preceq \arp_2$ and $\arp_1(\ell)=\arp_2(\ell)$, $\frac{\partial C(\arp_1,\up \xi,\down \xi)}{\partial\up \xi} \geq \frac{\partial C(\arp_2,\up \xi,\down \xi)}{\partial\up \xi}$.

\item For $a>0$, $\frac{\partial C(\arp+a,\up \xi,\down \xi)}{\partial\up \xi} \le \frac{\partial C(\arp,\up \xi,\down \xi)}{\partial\up \xi}$.

\item $\frac{\partial C(\arp,\up \xi,\down \xi)}{\partial\down \xi}$ is decreasing in $\arp$.
\end{enumerate}
\end{lemma}

\begin{lemma}
  \label{lemma:continuous}
  Continuity of the derivatives:
 $\frac{\partial C(\arp,\up \xi,\down \xi)}{\partial \up \xi}$ and $\frac{\partial
C(\arp,\up \xi,\down \xi)}{\partial \down \xi}$ are continuous in $\arp$.
\end{lemma}
\begin{proof}
Since the proofs are similar, we only prove the continuity for $\frac{\partial C(\arp,\up \xi,\down \xi)}{\partial\up \xi}$.
Suppose it is not continuous and there exists $a_0>0$ and a sequence $\{\arp_n,n=1,2,\ldots\}$ in $\mathbb{D}$ such that, as $n \to \infty$, $d(\arp,\arp_n) \to 0$ but $\frac{\partial C(\arp_n,\up \xi,\up \xi)}{\partial\up \xi}-\frac{\partial C(\arp,\up \xi,\down \xi)}{\partial\up \xi}>2a_0$ or $<-2a_0$ for all $n$. 
By the continuity in Lemma~\ref{lemma:continuous}, there exists $b_0>0$ such that $\frac{\partial C(\arp,\up \xi+b_0,\down \xi)}{\partial\up \xi} < \frac{\partial C(\arp,\up \xi,\down \xi)}{\partial\up \xi}+a_0$.
Thus, $\frac{\partial C(\arp,\up\xi+b_0,\down \xi)}{\partial\up \xi} < \frac{\partial C(\arp_n,\up \xi,\down \xi)}{\partial\up \xi}-a_0$.
Since $C(\arp,\up \xi,\down \xi)$ is convex in $\up \xi$, the partial derivative $\frac{\partial C(\arp,\up \xi,\down \xi)}{\partial\up \xi}$ is increasing in $\up \xi$.
We have
\begin{eqnarray*}
  C(\arp,\up \xi+b_0,\down \xi)-C(\arp,\up \xi,\down \xi)
&=&    \int_0^{b_0} \frac{\partial   C(\arp,\up \xi+s,\down \xi)}{\partial\up \xi}ds \\
&\leq& \int_0^{b_0}\frac{\partial C(\arp,\up \xi+b_0,\down \xi)}{\partial\up \xi}ds \\
&\leq& \int_0^{b_0} \left(\frac{\partial C(\arp_n,\up \xi,\down \xi)}{\partial\up \xi}-a_0\right)ds\\
&\leq& \int_0^{b_0} \frac{\partial C(\arp_n,\up \xi+s,\down \xi)}{\partial\up \xi} ds-\int_0^{b_0} a_0 ds \\
 &=&  C(\arp_n,\up \xi+b_0,\down \xi)-C(\arp_n,\up \xi,\up \xi)-a_0b_0.
\end{eqnarray*}

On the other hand, because $d(\arp,\arp_n)\to 0$ as $n\to\infty$ and $C^*(\arp)$ is continuous in $\mathbb{D}$,  $C(\arp_n,\up \xi+b_0,\down \xi)-C(\arp_n,\up \xi,\down \xi)$ converges to $C(\arp,\up \xi+b_0,\down \xi)-C(\arp,\up \xi,\down \xi)$ as $n\to\infty$. This is a contradiction and $\frac{\partial C(\arp,\up \xi,\down \xi)}{\partial\up \xi}$ is continuous in $\arp$. 
\end{proof}

We also note that $C(\arp,\up\xi,\down\xi)=C^*(\Phi_{\up \xi,\down\xi}(\arp)) 
\leq C^*(\Phi_{\up \xi+\up\epsilon,\down\xi+\down\epsilon}(\arp)) 
+\phi(\up\xi+\up \epsilon,\down\xi+\down\epsilon)=C^*(\arp,\up \xi+\up\epsilon,\down \xi+\down\epsilon)$ for any $\up\epsilon,\down\epsilon>0$. 
Thus, $\frac{\partial C(\arp,\up \xi,\down \xi)}{\partial \up \xi} \geq 0$ and $\frac{\partial C(\arp,\up \xi,\down \xi)}{\partial \down \xi}\geq 0$, and we have the following lemma.
\begin{lemma}
   \label{lemma:nonnegative}
   Non-negativity of the derivatives:
 $\frac{\partial C(\arp,\up \xi,\down \xi)}{\partial \up \xi}$ and $\frac{\partial
C(\arp,\up \xi,\down \xi)}{\partial \down \xi}$ are non-negative.
\end{lemma}
Since there are no fixed adjustment costs, any adjustment at a particular time can be viewed as the result of multiple simultaneous adjustments. 
Thus, starting with a smaller adjustment allows more flexibility and results in the non-negativity of the derivatives.

\subsection{The Set of Naturally Reachable States and Its Representation}
\label{sec:PDE}
Starting from an initial state $\arp$, the state at time $s$ will be $\ts_s(\arp)+w$ without any adjustment given a realization of the Brownian motion $W_s=w$. 
Thus, for any $s>0$ and $w \in \R$, we call $\ts_s(\arp)+w$ a \emph{naturally reachable state} from $\arp$ and $\{\ts_s(\arp)+w:s>0, w\in \R\} \subseteq \mathbb {D}$ is the set of all naturally reachable states from $\arp$.
For a fixed initial state $\arp$, any naturally reachable state can be fully described by a pair $(w,s)\in\R\times\R_+$, referred to as a reachable state from a given initial state with a slight abuse of notation.

\subsubsection{The Set of States where no Adjustment is Needed}
\label{sec:up-and-down-w}
At any naturally reachable state $(w,s)$ from an initial state $\arp$, an adjustment may or may not be needed. 
It is obvious that no upward (downward) adjustment should be made at $\arp$ if $\frac{\partial C(\arp,0,0)}{\partial\up \xi}> 0$ $\left(\frac{\partial C(\arp,0,0)}{\partial\down \xi}> 0 \right)$. 
That is, the set of naturally reachable states in which no adjustment is needed is given by 
\begin{equation}
\label{equ:no-control-region}
\Xi_{\arp} 
  = \left\{
    (w,s)\in\R\times\R_+:
    \frac{\partial C(\ts_s(\arp)+w,0,0)}{\partial \up\ja}>0, \
    \frac{\partial C(\ts_s(\arp)+w,0,0)}{\partial \down\ja}>0 
    \right\}.
\end{equation}
Let \begin{eqnarray}
  \label{eq:bound-H}
  \up w_{\arp}(s) &=& \max\left\{
  w\in\R: \frac{\partial C(\ts_s(\arp)+w,0,0)}{\partial\up \xi}=0
  \right\},\\
  \label{eq:bound-L}
  \down w_{\arp}(s) &=& \min\left\{
  w\in\R:\frac{\partial C(\ts_s(\arp)+w,0,0)}{\partial\down \xi}=0
  \right\}.
\end{eqnarray}
By Lemma~\ref{prop:mono-boundary-con}, $\frac{\partial C(\ts_s(\arp)+w,0,0)}{\partial\up \xi}>0$ $\left(\frac{\partial C(\ts_s(\arp)+w,0,0)}{\partial\down \xi}>0\right)$ if and only if $w>\up w_\arp(s)$ $\left(w<\up w_\arp(s)\right)$. 
Thus, \eqref{equ:no-control-region} is equivalent to
\begin{equation*}
\Xi_{\arp} 
  = \left\{
    (w,s)\in\R\times\R_+:
    \up w_{\arp}(s) < w < \down w_{\arp}(s)
    \right\}.
\end{equation*}
Since $\ts_s(\arp)$ increases in $s$ initially and remains   constant when $s\geq \ell$, by Lemma~\ref{prop:mono-boundary-con}, $\up w_{\arp}(s)$ increases in $s$ and stays constant at $\up w_\0(\lt)+\arp(\lt)$ for $s \geq \lt$ and $\down w_{\arp}(s)$ decreases in $s$ and stays constant at $\down w_\0(\lt)+\arp(\lt)$ for $s \geq \lt$ as shown in Figure~\ref{Figure:curves}. 

\begin{figure}[htb]
\begin{tikzpicture}[scale=0.49]
\def\lt{7}
\def\length{16.1}
\def\width{10}

\begin{scope} 
\coordinate (o) at (-3.1,-2);
\coordinate (o') at (0,0);
\draw [->] (-3.1,-3) -- +(\length,0) node [below] {$s$};
\draw [->] (-3.1,-3) -- +(0,\width) node [right] {$w$};
\draw plot[smooth] coordinates {(-3.1,-1.8) (-2,-1.1) (0,-0.4) (1,0.1) (2,0.6) (3.9,1)};
\draw plot[smooth] coordinates {(-3.1,5) (-2,4.2) (0,3.2) (1,2.9) (2,2.6) (3.9,2.2)};
\draw (3.9,2.2) -- +(8,0);
\draw (3.9,1) -- +(8,0);
\draw [dashed] ($(o)+(\lt,-1)$) -- ($(o)+(\lt,2.8)$);
\node [below] at ($(o)+(\lt,-1)$) {$\ell$};
\node [below] at (5.6,3.71) {$\down w_\arp(s)$};
\node [below] at (5.6,0.84) {$\up w_\arp(s)$};
\node [below] at (3.5,2.24) {$\Xi_\arp$};
\node [below] at (5.46,5.6) {downward};
\node [below] at (5.46,-1.4) {upward};
\node [below] at (-0,1.61) {no adjustment};
\end{scope}

\begin{scope} [xshift=18.2cm]
\coordinate (o) at (-3.1,-2);
\coordinate (o') at (0,0);
\draw [->] (-3.1,-3) -- +(\length,0) node [below] {$s$};
\draw [->] (-3.1,-3) -- +(0,\width) node [right] {$w$};
\draw plot[smooth] coordinates {(-3.1,-1.8) (-2,-1) (0,0) (1,0.84) (1.3,1.1) (2,1.6) (3.9,2.1)};
\draw plot[smooth] coordinates {(-3.1,5) (-2,4.2) (0,2.2) (1,1.4) (1.3,1.1) (2,0.7) (3.9,0)};
\draw (3.9,0) -- +(8,0);
\draw (3.9,2.1) -- +(8,0);
\draw [dashed] ($(o)+(\lt,-1)$) -- ($(o)+(\lt,2)$);
\node [below] at ($(o)+(\lt,-1)$) {$\ell$};
\node [below] at (0,4.34) {$\down w_\arp(s)$};
\node [below] at (0,-0.63) {$\up w_\arp(s)$};
\node [below] at (-1.4,2.38) {$\Xi_\arp$};
\node [below] at (5.46,4.2) {downward};
\node [below] at (5.46,-0.7) {upward};
\node [below] at (8.4,1.75) {downward and upward};
\node [below] at (-1.4,1.12) {no adjustment};
\end{scope}

\end{tikzpicture}
 \captionof{figure}{$\up w_\arp(s)$ and $\down w_\arp(s)$ which define $\Xi_\arp$} 
\label{Figure:curves}
\end{figure}
At any given state $\arp$, no adjustment is needed if $(0,0) \in \Xi_\arp$, or equivalently $\down w_\arp(0)<0<\up w_\arp(0)$. 
Otherwise, as $\up\xi$ ($\down \xi$) increases, by Lemma~\ref{prop:mono-boundary-con}, the marginal cost remains zero initially, i.e., $\frac{\partial C(\arp,\up\xi,\down \xi)}{\partial\up \xi}$ $\left(\frac{\partial C(\arp,\up\xi,\down \xi)}{\partial\down \xi}\right)$ stays at $0$ for a while until it becomes positive.
Since there are  no fixed control costs, intuitively, the optimal upward (downward) adjustment should be obtained at the maximum $\up\xi$ ($\down \xi$) at which the derivative is zero.
This means that an upward (downward) adjustment is needed at time $s$ if $w<\up w_\arp(s)$ ($w>\down w_\arp(s)$) as depicted in Figure~\ref{Figure:curves} and simultaneous upward and downward adjustments are needed at a reachable state $(w,s)$ if and only if $\down w_\0(\lt) \geq \up w_\0(\lt)$ as shown in the second case in Figure~\ref{Figure:curves}. 

Furthermore, we show in the next proposition that $\up w_\0(\lt)$ and $\down w_\0(\lt)$  provide sufficient information for deciding whether or not an upward or downward adjustment is not needed at a state.

\begin{proposition}
\label{prop:two-apart-bounds}
No upward and downward adjustment is needed at $\arp$ if $\arp(\lt) >\up w_\0(\lt)$ and $\arp(\lt) < \down w_\0(\lt)$, respectively.

\end{proposition}
\begin{proof}If $\arp(\lt)< \down w_{\0}(\lt)$, we have $\arp \preceq \down w_{\0}(\lt) {\mathcal I}$.
By part~3 of Lemma~\ref{prop:mono-boundary-con}, $\frac{\partial C(\arp,0,0)}{\partial\down \xi} \geq \frac{\partial C(\down w_{\0}(\lt){\mathcal I},0,0)}{\partial\down \xi}=0$.
Hence, by the definition of $\down w_{\0}(\lt)$, $\frac{\partial C(\arp,0,0)}{\partial\down\xi}>0$.
If $\arp(\lt) > \up w_{\0}(\lt)$, there exists $b>0$ such that $\arp(\lt)-b=\up w_{\0}(\lt)$ and $\arp-b \preceq \down w_{\0}(\lt){\mathcal I}$.
Then $\frac{\partial C(\arp,0,0)}{\partial\up \xi} 
\geq \frac{\partial C(\arp-b,0,0)}{\partial\up \xi} 
\geq \frac{\partial C(\up w_{\0}(\lt){\mathcal I},0,0)}{\partial\up \xi}=0$
by parts~2 and 1 of Lemma~\ref{prop:mono-boundary-con}, and $\frac{\partial C(\arp,0,0)}{\partial\up \xi}>0$ by the definition of $\up w_{\0}(\lt)$.
\end{proof}

\subsubsection{Properties of $V_{\arp}(w,s)$ on $\Xi_\arp$} 
In this section, we will show that the optimal value function $V_{\arp}(w,s)$ defined in \eqref{eq:cost-no-control} is a solution to the HJB equation \eqref{eq:optimality-heuristic} with  $\Xi_\arp$ as the boundaries and identify the timing of adjustment.

\begin{theorem}
  \label{thm:PDE}
  The partial derivatives $\frac{\partial V_{\arp}(w,s)}{\partial s}$, $\frac{\partial V_{\arp}(w,s)}{\partial w}$ and $\frac{\partial^2 V_{\arp}(w,s)}{\partial w^2}$ exist and
  \begin{equation}
    \label{equ:PDE}
    \frac{\partial V_{\arp}(w,s)}{\partial s}
    +\frac{\sigma^2}{2}\frac{\partial^2 V_{\arp}(w,s)}{\partial w^2}
    +\mu\frac{\partial V_{\arp}(w,s)}{\partial w}-\gamma V_{\arp}(w,s)
    +h(\arp(s)+w)=0
  \end{equation}
holds for almost every $(w,s)\in \Xi_{\arp}$.
\end{theorem}

The key step  in proving Theorem~\ref{thm:PDE} is to establish the following property of $V_{\arp}(w,s)$ in a small enough neighborhood of any point $(\hat{w},\hat{s})$ in $\Xi_{\arp}$.

\begin{proposition}
\label{lemma:small-neighbour}
For a given $\arp \in \mathbb{D}$ and $(\hat{w},\hat{s})\in \Xi_{\arp}$, there exists a neighbourhood of $(\hat{w},\hat{s})$, $\Xi^{(\hat{w},\hat{s})}_\arp \subset \Xi_{\arp}$, such that \begin{equation}
   \label{equ:Int-PDE-small}
    V_{\arp}(w,s)
    =\E\left[\int_0^{\spt} e^{-\gamma t}h(\arp(s+t)+w+W_t) dt\right]
    +\E\big[e^{-\gamma \spt}V_{\arp}(w+W_{\spt},s+\spt)\big],
  \end{equation}
for all $(w,s) \in \Xi^{(\hat{w},\hat{s})}_\arp$, where $\spt$ is the first time the process $\{(w+W_{t},s+t): t \geq 0\}$ leaves $\Xi^{(\hat{w},\hat{s})}_\arp$.
\end{proposition}

\begin{proof}
By the continuity of the partial derivatives in Lemma~\ref{lemma:continuous}, there exist $\delta>0$ and $k_0>0$ such that, for any $\hat\arp$ satisfying $d(\ts_{\hat{s}}(\arp)+\hat{w},\hat{\arp})< 3\delta$, 
\begin{eqnarray}
  \frac{\partial C(\hat{\arp},0,0)}{\partial \up\ja} \geq k_0, 
  \quad \frac{\partial C(\hat{\arp},0,0)}{\partial \down\ja} \geq k_0.     \label{eq:subregion}
\end{eqnarray}
Consider a neighbourhood of $(\hat{w},\hat{s})$, $\Xi^{(\hat{w},\hat{s})}_\arp=(\hat{w},\hat{s})+B(\delta)$, where $B(\delta):=[-\delta,\delta] \times \left[0,\frac{\delta}{\gamma \arp(\lt)} \right]$. 
For any $(w,s) \in \Xi^{(\hat{w},\hat{s})}_\arp$, recall the definition of the distance $d(\cdot,\cdot)$ in Section~\ref{sec:model-formulation},
\begin{eqnarray*}
d(\ts_{\hat{s}}(\arp)+\hat{w},\ts_{s}(\arp)+w)
&\leq& d(\ts_{\hat{s}}(\arp)+\hat{w},\ts_{s}(\arp)+\hat{w})
       +d(\ts_{s}(\arp)+\hat{w},\ts_{s}(\arp)+w)\\
&\leq& (s-\hat{s}) \gamma \arp(\lt) + \delta \leq 2\delta.
\end{eqnarray*}
Thus, $\Xi^{(\hat{w},\hat{s})}_\arp \subset \Xi_{\arp}$.
Next, show the proposition holds in this neighborhood by contradiction via the following three steps. 

\begin{enumerate}
 \item Suppose \eqref{equ:Int-PDE-small} does not hold at a pair $(w',s')\in \Xi^{(\hat{w},\hat{s})}_\arp$.
Since $V_{\arp}(w',s')=C^{*}(\ts_{s'}(\arp)+w')$ is the optimal cost at $\ts_{s'}(\arp)+w'$, there exists a positive $c_0$ such that  
\begin{equation}
  \label{eq:inequ-small-neighbour}
  \E\left[\int_0^{\tau'} e^{-\gamma t}h(\mathcal{X}(s'+t)+w'+W_t) dt\right]
  +\E\big[e^{-\gamma \tau'} V_{\arp}(w'+W_{\spt'},s'+\spt')\big] >V_{\arp}(w',s')
+c_0,
\end{equation}
where $\tau'$ is the stopping time when $\{(w'+W_{t},s'+t): t \geq 0\}$ leaves $(\hat{w},\hat{s})+B(\delta)$.
Introducing $\arp'=\ts_{s'}(\arp)+w'$, \eqref{eq:inequ-small-neighbour} is equivalent to  
\begin{equation}
  \label{eq:o-inequ-small-neighbour}
  \E\left[\int_0^{\spt'} e^{-\gamma t}h(\mathcal{X'}(t)+W_t) dt\right]
  +\E\big[e^{-\gamma \tau'} V_{\arp'}(W_{\spt'},\spt')\big] >V_{\arp'}(0,0)+c_0.
\end{equation}
For any feasible periodic control policy $\pi^n=(Y^{n\uparrow},Y^{n\downarrow})$ defined in \eqref{eqn:periodic-policy}, let $\arp'_t$ be the updated state at time $t$ under this policy $\pi^n$. 
For any $\epsilon \leq \delta$, we define $N(\epsilon)=\inf\{k:\sum_{i=0}^k (\xi^{n\uparrow}_i+\xi^{n\downarrow}_i) \geq \epsilon \}$.
Without loss of generality, assume that $\sum_{i=0}^{N(\epsilon)} (\xi^{n\uparrow}_i+\xi^{n\downarrow}_i)=\epsilon$, otherwise  split the adjustments $\xi^{n\uparrow}_{N(\epsilon)}$ and/or $\xi^{n\downarrow}_{N(\epsilon)}$ into two. 
Let $A$ be the event where the $N(\epsilon)$th adjustment is made after $\tau'$, i.e., $A=\{\tau'\le T^n_{N(\epsilon)}\}$.

\item
Estimate the cost $C(\arp',\pi^n)$ by considering the events $A$ and $A^c$, respectively.

(a) On the event $A$, from \eqref{eq:subregion},  the marginal costs for the upward and downward adjustments are quite large  implying that 
\begin{equation}
  \label{equ:A happens}
  C(\arp',\pi^n) \geq C^*(\arp') +\prob(A) e^{-\gamma \delta} k_0 \epsilon
              \geq V_{\arp'}(0,0)+ \prob(A) e^{-\gamma \delta} k_0 \epsilon.
\end{equation}
(b) On the event $A^c$, the cumulative amount of upward and downward adjustment by the stopping time $\tau'$ is less than $\epsilon$, i.e. $d(\arp'_{s},\ts_{s}(\arp')+W_{s}) < \epsilon$ for any $0 \le s \le \tau'$.
Combining with \eqref{eq:o-inequ-small-neighbour} we can imply that
\begin{eqnarray}
  C(\arp',\pi^n) \geq
  V_{\arp'}(0,0) +c_0-\frac{M \epsilon}{\gamma}-\prob(A) (\delta\bar{h}+\bar{V}),
  \label{equ:A^c happens}
\end{eqnarray}
where $\bar{h}$ and $\bar{V}$ are two constants.
The detailed proofs of \eqref{equ:A happens} and \eqref{equ:A^c happens} are presented in the Appendix.
\item
Properly choose $\epsilon=\min\{\delta,\frac{c_0 \gamma}{2M}\}$ and denote $p_0=\frac{c_0}{2e^{-\gamma \delta}k_0\min\{\delta,\frac{c_0 \gamma}{2M}\}+2\delta\bar{h}+2\bar{V}}$.
If $\prob(A) \geq p_0$, from \eqref{equ:A happens}  $C(\arp',\pi) \geq V_{\arp'}(0,0)+e^{-\gamma \delta}k_0 \min\{\delta,\frac{c_0 \gamma}{2M}\}p_0$.
Otherwise, if $\prob(A) \leq p_0$, from \eqref{equ:A^c happens} $C(\arp',\pi) \geq V_{\arp'}(0,0)+e^{-\gamma \delta} k_0 \min \{\delta,\frac{c_0 \gamma}{2M}\}p_0$.
Thus, for any discrete policy $\pi^n$, its associated expected cost will be at least $V_{\arp'}(0,0)+e^{-\gamma \delta} k_0 \min \{\delta,\frac{c_0 \gamma}{2M}\}p_0$.
However, this is a contradiction of Proposition~\ref{prop:C^n converge to C}.
\end{enumerate}
\end{proof}
We are now ready to prove Theorem~\ref{thm:PDE}.
\begin{proof}[Proof of Theorem~\ref{thm:PDE}]
In Proposition~\ref{lemma:small-neighbour}, we've proved that for any $(\hat{w},\hat{s})\in \Xi_{\arp}$, we can find a corresponding subset $\Xi^{(\hat{w},\hat{s})}_\arp \in \Xi_{\arp}$ such that all the points in the subset satisfy \eqref{equ:Int-PDE-small}.
Applying  Dynkin's law in \cite{Dynkin1956} to \eqref{equ:Int-PDE-small}, we find that \eqref{equ:PDE} holds for all $(w,s)$ in $\Xi^{(\hat{w},\hat{s})}_\arp$.
In other words, for any $(\hat{w},\hat{s})\in \Xi_{\arp}$, we can find a corresponding neighborhood $\Xi^{(\hat{w},\hat{s})}_\arp \in \Xi_{\arp}$ where \eqref{equ:PDE} holds.
Since $\Xi_{\arp}=\bigcup \limits_{(w,s) \in \Xi_{\arp}} \Xi^{(w,s)}_\arp$, we can conclude that \eqref{equ:PDE} holds for all points in $\Xi_{\arp}$.
\end{proof}

Define ${\spt_{\arp}} \geq 0$ to be the first time the process $\{(w+W_{t},s+t): t \geq 0\}$ leaves $\Xi_{\arp}$. 
By Ito's formula and Theorem~\ref{thm:PDE}, we have the following corollary (whose proof is skipped as it is same to that of Theorem~\ref{thm:PDE}).
The corollary helps to identify the time of the adjustment since the equation in the corollary actually holds for any stopping time $\tau$ such that $\tau \le \tau_{\arp}$ with probability $1$.

\begin{corollary}
\label{cor:general-ine-pde}
  For any given $\arp \in \mathbb{D}$ and $(w,s)\in \Xi_{\arp}$,
  \begin{equation*}
    V_{\arp}(w,s)
    =\E\left[\int_0^{\spt_{\arp}} e^{-\gamma t}h(\arp(s+t)+w+W_t) dt\right]
    +\E\big[e^{-\gamma {\spt_{\arp}}}V_{\arp}(w+W_{\spt_{\arp}},s+\spt_{\arp})\big].
  \end{equation*}
\end{corollary}

Based on the $L^\natural$-convexity of $C^*(\arp)$ and its optimality, we have the following proposition.
\begin{proposition}
  \label{prop:PDE-positive}
  The partial derivatives $\frac{\partial V_{\arp}(w,s)}{\partial s}$, $\frac{\partial V_{\arp}(w,s)}{\partial w}$ and $\frac{\partial^2 V_{\arp}(w,s)}{\partial w^2}$ exist and
  \begin{equation*}
    \frac{\partial V_{\arp}(w,s)}{\partial s}
    +\frac{\sigma^2}{2}\frac{\partial^2 V_{\arp}(w,s)}{\partial w^2}
    +\mu\frac{\partial V_{\arp}(w,s)}{\partial w}-\gamma V_{\arp}(w,s)
    +h(\arp(s)+w) \ge 0
  \end{equation*}
holds for almost every $(w,s) \in \R \times \R_+$.
\end{proposition}

The above proposition and Lemma~\ref{lemma:nonnegative} show that each one of the three terms of \eqref{eq:optimality-heuristic} is always non-negative. 
Moreover, if $\arp \in \Xi_\arp$, the first term of \eqref{eq:optimality-heuristic} must be zero by Theorem~\ref{thm:PDE}.
Otherwise, by the definition of $\Xi_\arp$, at least one of the last two terms of \eqref{eq:optimality-heuristic} is zero.
This yields the following theorem.

\begin{theorem}
\label{thm:HJB-equation}
   For any $\arp \in \mathbb{D}$, $V_{\arp}(w,s)$ is a solution to the HJB equation \eqref{eq:optimality-heuristic} with $\Xi_\arp$ as the boundaries.
\end{theorem}

\section{The Optimal Control Policy}
\label{sec:optimal-policy-zero}

In this section, we will construct an optimal control policy. 
We will first define the set of states in which an upward or  downward adjustment is needed. 
We then examine the corresponding upward (downward) adjustment policy for a given downward (upward) adjustment policy, referred to as the one-sided reflection mapping in Section~\ref{sec:A One-side Reflection Mapping of Varying Bounds}, and construct a two-sided reflection mapping in Section~\ref{sec:A Two-side Reflection Mapping of Varying Bounds}.
We then show that the two-sided reflection mapping is an optimal control in Section~\ref{sec:Optimality of the Two-Side Reflection Policy}. 

We have demonstrated in Figure~\ref{Figure:curves} that an upward (downward) adjustment is needed at time $s$ if $w<\up w_\arp(s)$ ($w>\down w_\arp(s)$). 
Define 
\begin{equation}
  \label{eq:control-regions'}
  \up {\mathbb D}=\left\{
     \arp \in {\mathbb D}: \up w_{\arp}(0)<0
   \right\} 
   \mbox{ and }
   \down{\mathbb D}=\left\{
     \arp \in {\mathbb D}: \down w_{\arp}(0)>0
     \right\}
\end{equation}
to be the subsets of $\mathbb D$ in which an upward adjustment and downward adjustment are needed, respectively. 
Lemma~\ref{prop:mono-boundary-con} immediately leads to the following corollary.
\begin{corollary}
\label{lemma:apply-der-to-H-L} 
Let $\up{\bar{\mathbb D}}$ and $\down{\bar{\mathbb D}}$ be the complements of $\up{\mathbb D}$ and $\down{\mathbb D}$, respectively. 
\begin{enumerate}
  \item If $\arp\succeq \arp'$ and $\arp(\lt)=\arp'(\lt)$, then $\arp'\in \up {\bar{\mathbb D}}$ implies $\arp\in \up {\bar{\mathbb D}}$.
  \item If $\arp \in \up {\bar{\mathbb D}}$, then $\arp+a \in \up {\bar{\mathbb D}}$
for all $a>0$.
  \item If $\arp\succeq \arp'$, then $\arp\in\down {\bar{\mathbb D}}$ implies $\arp'\in\down {\bar{\mathbb D}}$.
  \end{enumerate}
\end{corollary}

If there exists a state belonging to both $\up{\mathbb D}$ and $\down{\mathbb D}$, then we need to make downward and upward adjustments at the same time. 
This can happen when it is too costly to hold a unit of inventory that is likely to be needed $\lt$ amount of time later, i.e., when the cost for holding a large amount of inventory is relatively high and the lead time is relatively long.  
When this happens, the optimal adjustment can be quite complicated.
Thus, we will focus on the case where $\up {\mathbb D} \cap \down {\mathbb D}=\O$ which holds in most real applications.  
The following lemma also provides an explicit sufficient condition for this to hold. 

\begin{lemma}
\label{assumption:apart-1} 
$\up {\mathbb D} \cap \down {\mathbb D}=\O$ if and only if $\down w_{\0}(\lt)> \up w_{\0}(\lt)$. 
A sufficient condition for $\up {\mathbb D} \cap \down {\mathbb D}=\O$ is $\up k+\down k> \frac{1-e^{-\gamma \lt}}{\gamma}\max\limits_{x>0}h'(x)$.
\end{lemma}
\begin{proof}
A direct result from Figure \ref{Figure:curves} is that a necessary and sufficient condition for non-simultaneous upward and downward adjustments is $\down w_{\0}(\lt)>\up w_{\0}(\lt)$. 
If these two subsets intersect and  $(\up \xi, \down \xi)$ are simultaneously adjusted, for a downward and an upward adjustment $(\up \xi-\epsilon, \down \xi-\epsilon)$, we increase the holding cost by no more than $\epsilon\int_0^{\lt} e^{-\gamma t}\max\limits_{x>0}h'(x) dt$ while reducing  the control cost by $(\up k+\down k)\epsilon$.
Since $\up k+\down k>\max\limits_{x>0}h'(x) \frac{1-e^{-\gamma \lt}}{\gamma}$, the total cost will decrease.
\end{proof}

\subsection{Reflection mappings} 
\label{ReflectionMappings}
We first identify the minimum upward (downward) adjustment needed to ensure $\arp_s \in \up {\bar{\mathbb D}}$ ($\arp_s \in \down {\bar{\mathbb D}}$) at all $s \geq 0$ for a given downward (upward) adjustment.
We refer to them as one-sided reflection mappings that will lead to the two-sided reflection policy, an optimal control.

\subsubsection{One-sided Reflection Mappings}
\label{sec:A One-side Reflection Mapping of Varying Bounds}
For a given sample path of the Brownian motion $\omega$ and initial state $\arp$, the state $\arp_s$ under policy $(\up Y,\down Y)$ can also be written as
\begin{equation*}
\arp_s=\ts_s(\arp)+\omega (s)-\down Y(s)+\ts_{s-\lt}(\up Y)\wedge \up Y(s)\mathcal {I}
\end{equation*}
by the dynamics (\ref{equ:dynamics}). 
For convenience, we use the superscripts $i,j \in \{\uparrow,\downarrow\}$, $i \neq j$, to indicate a pair of upward and downward adjustments. 
For any given $(\arp,Y^j,\omega)$, 
\begin{equation}
\label{def:up-one-side}
\Pi^i(\arp,Y^j,\omega) =\{Y^i:\arp _s=\ts_s(\arp)-\down Y(s)+\ts_{s-\lt}(\up Y)\wedge \up Y(s)\mathcal {I}+\omega(s)\in \bar{\mathbb D}^i,\ \textrm{for all }s\ge 0\}  
\end{equation}
is the set of all the feasible one-sided adjustments $Y^i$ that will ensure $\arp_s \in \bar{\mathbb D}^i$ at all $s$.
Recall that $\mathbb D$ is a functional set. For any subset $\emptyset \neq \mathbb S \subseteq \mathbb D$, let $\inf \mathbb {S}$ be a function that takes the infimum of all functions in $\mathbb S$ at any point, i.e., 
\begin{equation*}
  (\inf \mathbb {S})(t)=\inf\limits_{f\in \mathbb S}\{f(t)\} \quad \textrm{for any } t\ge 0.
\end{equation*}
By Lemma~14.2.2 in \cite{Whitt2002}, $\inf \mathbb S \in \mathbb D$.

\begin{definition}[One-sided reflection mappings]
\label{def:one-side reflection mapping}
We call $\psi^i$: $(\mathbb D,\mathbb D,\mathbb D) \to \mathbb D$ a  one-sided reflection mapping for  ${\mathbb D}^i$ if, for a given state $\arp$, sample path $\omega$ and $Y^j \in\mathbb D,$
\begin{eqnarray*}
  \psi^i(\arp,Y^j,\omega)=\inf {\Pi^i}(\arp,Y^j,\omega).
\end{eqnarray*}
\end{definition}
 

Next, we show the existence of the one-sided reflection mappings in Proposition~\ref{prop:existence of one-side reflection mapping} and provide some properties of the mappings in Proposition~\ref{prop:order and continuous on one-side}.
\begin{proposition}[Existence of the reflection maps]  
\label{prop:existence of one-side reflection mapping}
 $\psi^i(\arp,Y^j,\omega)$ exists and belongs to $\Pi^i(\arp,Y^j,\omega)$.
\end{proposition}
\begin{proof}
Since the proofs are similar, we only prove the result for $\up \psi(\arp,\down Y,\omega)$. 
We first claim  that $\up \Pi(\arp,\down Y,\omega)$ is non-empty as an adjustment $g(t)=\sup\limits_{0 \le u \le t}\{\ulb-\omega(u)+\down Y(u)-\arp(u+\lt)\} \in \up \Pi(\arp,\down Y,\omega)$. 
This is because
\begin{equation*}
  \arp_s(\lt)=\arp(s+\lt)+\omega(s)-\down Y(s)+g(s) \geq \ulb, \ \textrm{for any } s \ge 0
\end{equation*}
and by Proposition~\ref{prop:two-apart-bounds}, $\frac{\partial C(\arp_s,0,0)}{\partial \up \xi}=0$. 
So $\up \Pi(\arp,\down Y,\omega)$ at least has one element. 

It remains to be shown that $\up \psi(\arp,\down Y,\omega) \in \up \Pi(\arp,\down Y,\omega)$. 
For any fixed $\epsilon>0$ and $s \ge 0$, there exists $Y^{\uparrow'} \in \up \Pi(\arp,\down Y,\omega)$ such that $Y^{\uparrow'} \succeq \up \psi(\arp,\down Y,\omega)$ and $Y^{\uparrow'}(s) \leq \up \psi(s)+\epsilon$.
Thus, $\ts_s(\arp)+w(s)-\down Y(s)+\ts_{s-\lt}(Y^{\uparrow'})\wedge Y^{\uparrow'}(s)\mathcal I \in \up {\bar{\mathbb D}}$. 
Then, by parts~1 and 2 of Corollary~\ref{lemma:apply-der-to-H-L}, we know that 
$\ts_s(\arp)+w(s)-\down Y(s)+\ts_{s-\lt}(\up \psi(s)+\epsilon)\wedge (\up \psi(s)+\epsilon) \mathcal I \in \up {\bar{\mathbb D}}$.
Because $s$ and $\epsilon$ are arbitrary, $\up \psi(\arp,\down Y,\omega) \in \up \Pi(\arp,\down Y,\omega)$.
\end{proof}

\begin{proposition}
\label{prop:order and continuous on one-side}
Let $\arp$ be the initial state and $\omega$ a sample path. 
\begin{enumerate}
\item 
$\up \psi(\arp,\down Y,\omega)$ decreases in $\down Y$ and $\down \psi(\arp,\up Y,\omega)$ increases in $\up Y$. 

\item 
$\sup\limits_{0 \le u \le t}|\psi^i(\arp,Y^j_1,\omega)(u)-\psi^i(\arp,Y^j_2,\omega)(u)|  \leq \sup\limits_{0 \le u \le t}|Y^j_1(u)-Y^j_2(u)|$ for any given $t \ge 0$, hence $\psi^i(\arp,Y^j,\omega)$ is Lipschitz continuous in $Y^j$ under the uniform norm.
\end{enumerate}
\end{proposition}

\begin{proof}
We will only prove the results for $\up \psi(\arp,\down Y,\omega)$.
\begin{enumerate}
\item 
Suppose $\down Y_1 \succeq \down Y_2$.
For any  $Y \in\up \Pi(\arp,\down Y_1,\omega)$, $\ts_s(\arp)+\ts_{(s-\lt)}(Y) \wedge Y(s)\mathcal I-\down Y_1(s)+\omega(s) \in \up {\bar{\mathbb D}}$. 
By part 2 of Corollary~\ref{lemma:apply-der-to-H-L}, $\ts_s(\arp)+\ts_{(s-\lt)}(Y )\wedge Y(s)\mathcal I-\down Y_2(s)+\omega(s) \in \up {\bar{\mathbb D}}$ for all $s \ge 0$ and consequently $Y \in \up \Pi(\arp,\down Y_2,\omega)$.
Thus, $\Pi(\arp,\down Y_1,\omega) \subseteq \Pi(\arp,\down Y_2,\omega)$ and $\up \psi(\arp,\down Y_1,\omega)  \preceq \up \psi(\arp,\down Y_2,\omega)$.
\item
We prove this part by contradiction. 
For convenience, let  $a_0=\sup\limits_{0 \le u \le t}|\down Y_1(u)-\down Y_2(u)|<\infty $ and  $g_1=\up \psi(\arp,\down Y_1,\omega)$ and $g_2=\up \psi(\arp,\down Y_2,\omega)$.
Suppose that the inequality does not hold. 
Define $\tau:=\inf \{s \geq 0: |g_2(s)-g_1(s)|>a_0\}$.
Without loss of generality, we assume $g_2(\tau) \ge g_1(\tau)+a_0$.
Because $g_1$ and $g_2$ are right-continuous, there exists an $\epsilon < \lt$ such that $g_2(s)-g_1(s)> a_0$  for $s \in (\tau,\tau+\epsilon]$. 
Consider the following function
\begin{equation*}
  g_2'(u)=\left\{ 
 \begin{array} {ll}
  g_1(u)+a_0 & u \in [\tau,\tau+\epsilon),\\
  g_2(u) & \textrm{ otherwise }.
  \end{array}
  \right.
\end{equation*} 
Then, for all $t<\tau$,  $g'_2(t)=g_2(t) \le g_1(\tau)+a_0=g'_2(\tau)$ and $g'_2(\tau+\epsilon)=g_2(\tau+\epsilon) > g_1(\tau+\epsilon)+a_0$. 
Thus,  $g'_2$ is also non-decreasing and strictly less than $g_2$. Next, we show that $g'_2\in\Pi(\arp,\down Y_2,\omega)$ or equivalently, for all $s\ge 0$,
\begin{equation}
  \ts_s(\arp)+\ts_{(s-\lt)}g_2' \wedge g'_2(s)\mathcal I-\down Y_2(s)+\omega(s) \in
\up {\bar{\mathbb D}}
\label{eq:modify-g}
\end{equation} 
and hence, we have a contradiction.  
Note that, $\ts_s(\arp)+\ts_{(s-\lt)}g_k \wedge g_k(s)\mathcal I-\down Y_k(s)+\omega(s) \in \up {\bar{\mathbb D}}$ for $k=1,2$.
\begin{itemize}
\item For $0 \le s < \tau$, $\ts_{(s-\lt)}(g_2') \wedge g'_2(s)\mathcal I=\ts_{(s-\lt)}(g_2) \wedge g_2(s)\mathcal I$, and \eqref{eq:modify-g} holds.

\item For $ \tau \le s \le \tau+\epsilon$,  $g_2'(s)=g_1(s)+a_0$ and $\ts_{(s-\lt)}(g_2') \wedge g'_2(s)\mathcal I+a_0 \succeq \ts_{(s-\lt)}(g_1) \wedge g_1(s)\mathcal I$.
By part 1 of Corollary~\ref{lemma:apply-der-to-H-L},
$\ts_s(\arp)+\ts_{(s-\lt)}(g_2') \wedge g'_2(s)\mathcal I-a_0-\down Y_1(s)+\omega (s) \in \up {\bar{\mathbb D}}$.
Since $a_0+\down Y_1(s) \ge \down Y_2(s)$, \eqref{eq:modify-g} holds by part 2 of Corollary~\ref{lemma:apply-der-to-H-L}.

\item For $s > \tau+\epsilon$,  $g_2'(s)=g_2(s)$ and $\ts_{(s-\lt)}(g_2') \wedge g'_2(s) \mathcal I\preceq \ts_{(s-\lt)}(g_2) \wedge g_2(s) \mathcal I$. 
By part 1 of Corollary~\ref{lemma:apply-der-to-H-L}, \eqref{eq:modify-g} holds.
\end{itemize}
\end{enumerate}
\end{proof}

Due to the ``$\inf$'' operator, $\psi^i(\arp,Y^j,\omega)(t)$ increases in $t$ only when $\arp_t$ hits the boundary of ${\mathbb D}^i$, i.e., $\frac{\partial C(\arp_t,0,0)}{\partial \xi^i} =0$, which is summarized in the following proposition. 
\begin{proposition}[Complementarity of the reflection mappings]
  \label{prop:characterization of the mapping}
  If $\arp_t$ is the state at time $t$ under policy $\psi^i$ for a given $Y^j$ and initial state $\arp$, then $\int_a^{b} \frac{\partial C(\arp_t,0,0)}{\partial \xi^i} d \psi^i(\arp,Y^j,\omega)(t)=0$ for any $0 \leq a \leq b \leq \infty$.
\end{proposition}

\subsubsection{A Two-sided Reflection Mapping}
\label{sec:A Two-side Reflection Mapping of Varying Bounds}

We are now ready to define a two-sided reflection mapping, and show its existence and uniqueness.
\begin{definition}[A two-sided reflection mapping]
\label{def:two-side reflection mapping}
For a given $\arp$ and Brownian motion sample path $\omega$, $(\up Y,\down Y)$ is called a two-sided reflection mapping if 
  \begin{eqnarray}
  && \up Y=\up \psi(\arp,\down Y,\omega), \label{ite:down-to-up} \\
  && \down Y=\down \psi(\arp,\up Y,\omega). \label{ite:up-to-down}
   \end{eqnarray}
  \end{definition}
\begin{proposition}
\label{prop:existence of two-side reflection mapping}
For any given $\arp$ and Brownian motion sample path $\omega$, there exists a unique two-sided reflection mapping $(Y^{\uparrow*},Y^{\downarrow*})$.
\end{proposition}
\begin{proof}
\ \underline{The existence of a two-sided mapping:}
We show the existence of a two-sided mapping as the limit of a series of one-sided mappings and the convergence of the mappings is achieved in a finite number of steps. 
For any $\arp$ and sample path $\omega$, we construct a series of upward and downward reflection mappings as $\down Y_0=\0$ and 
\begin{eqnarray}      
&&\up Y_k=\up \psi(\arp,\down Y_{k-1},\omega), \label{eq:ite_up}\\
&&\down Y_k=\down \psi(\arp,\up Y_k,\omega),\label{eq:ite_down}
\end{eqnarray}
for $k=1,2,3,\cdots$.
By part 1 of Proposition~\ref{prop:order and continuous on one-side}, one can easily see that both $\up Y_k$ and $\down Y_k$ increase in $k$ (in the sense of ``$\preceq$'') and hence converge as $k\to\infty$. 
We now show that, for any fixed $t$, both $\up Y_k(t)$ and $\down Y_k(t)$ converge in a finite number of steps.

Let $\arp^{k \uparrow }_s \in \up{\bar{\mathbb D}}$ denote the resulting state at time $s$ under policy $(\up Y_{k},\down Y_{k-1})$ and $\arp^{k \downarrow}_s \in \down{\bar{\mathbb D}}$ denote the state at time $s$ under policy $(\up Y_k,\down Y_k)$, for $k=1,2,\cdots$.
Let $\up t_k=\inf\{t:\arp_t^{k \uparrow} \in \down{\mathbb D}\}$ and $\down t_k=\inf\{t:\arp_t^{k \downarrow} \in \up{\mathbb D}\}$ be the first time $\arp^{k \uparrow}_t$ enters $\down{\mathbb D}$ and $\arp^{k \downarrow}_t$ enters $\up{\mathbb D}$, respectively.

We first prove that for any given $k \geq 1$, $\down Y_m=\down Y_{k-1}$ on $[0,\up t_k]$ for all $m\ge k$. The proof of $\up Y_m=\up Y_k$ on $[0,\down t_k]$ for all $m \geq k$ is similar and hence omitted. 
Thus, $\up Y_m$ and $\down Y_m$ converge to $\up Y_k $ and $\down Y_k$ in $k$ steps on $[0,\up t_{k+1}]$ and $[0,\down t_k]$, respectively. 
Since $\down Y_{k-1}(s) \leq \down Y_{k}(s) =\down \psi(\arp,\up Y_{k},\omega)(s)$ for all $s \ge 0$, $\down Y_{k-1}$ is a smaller downward adjustment than $\down Y_k$ and can also prevent the profile from entering $\down{\mathbb D}$ as $\arp^{k \uparrow}_s \in \down{\bar{\mathbb D}}$ for $s\in [0,\up t_k]$. 
Note that the one side mappings \eqref{eq:ite_up} and \eqref{eq:ite_down} on $[0,s]$ only depend on the sample path $\omega$ on $[0,s]$. 
Thus, $\down Y_{k-1}=\down Y_k$ on $[0,\up t_k]$ implying that $(\up Y_k,\down Y_{k-1})$ jointly satisfy \eqref{ite:down-to-up} and \eqref{ite:up-to-down} on $[0,\up t_k]$. 
Hence, $\up Y_m=\up Y_k$ and $\down Y_m=\down Y_{k-1}$ on $[0,\up t_k]$ for $m \geq k$.

Next, we show that $\up t_k \le \down t_k\le \up t_{k+1}\le \down t_{k+1}$  for any given $k\ge 1$.
Since $\down Y_k=\down Y_{k-1}$ on $[0,\up t_k], $ $\arp^{k \downarrow}_s=\arp^{k \uparrow}_s$ for  $s\in[0,\up t_k]$ and $\up t_k \le \down t_k$.
Likewise, since $\up Y_{k+1}=\up Y_k$ on $[0,\down t_k], $ $\down t_k\le \up t_{k+1}$. 

It remains for us to show that, for any fixed $t$, there exists $k'$ such that $\down t_{k'} \geq t$. 
Denote $s_{k}=\inf\big\{\up t_k \leq t \leq \down t_k:\down Y_{k}(t)=\down Y_{k}(\down t_k)\big\}$.
Then, no adjustment is made on $[s_{k},\down t_k]$ and $\arp^{k \downarrow}_t$ enters $\up{\mathbb D}$ after $\down t_k$ due to the Brownian motion $\omega$.
If $\down Y_k(s_k)>\down Y_k(s_k-)$, by Proposition~\ref{prop:characterization of the mapping}, $\frac{\partial C(\arp_{s_k}^{k\downarrow},0,0)}{\partial \down \xi}=0$.
On the other hand, if $\down Y_k(s_k)=\down Y_k(s_k-)$, we can find an increasing sequence $\{u_p,p=1,2,\cdots\}$ such that $\lim\limits_{p \to \infty}u_p=s_k$ and $\down Y_{k}$ increases at $u_p$.
By Proposition~\ref{prop:characterization of the mapping}, we have $\frac{\partial C(\arp_{u_p}^{k\downarrow},0,0)}{\partial \down \xi}=0$ for $p=1,2,\ldots$, which implies that $\frac{\partial C(\arp_{s_k}^{k\downarrow},0,0)}{\partial \down \xi}=0$ following the continuity property in Lemma~\ref{lemma:continuous}.
Then, by Proposition~\ref{prop:two-apart-bounds}, we must have
$\arp(s_{k})+\omega(s_{k})-\down Y_{k}(s_{k})+\up Y_{k}(s_{k}) \geq \uub$ and $\arp(\down t_k)+\omega(\down t_k)-\down Y_{k}(\down t_k)+\up Y_{k}(\down t_k) \leq \ulb$.

Since $\down Y_{k}(s_{k})=\down Y_{k}(\down t_k)$, $\arp(s_k)\le \arp(\down t_k)$ and $\up Y_{k}(s_k) \le \up Y_k(\down t_k)$, we have 
\begin{equation}
\label{eq:tech-fluc-W}
\omega({s_{k}})-\omega(\down t_k) \geq \uub-\ulb.
\end{equation} 
By the continuity of the sample path $\omega$, there exists a $\delta>0$ such that
\begin{eqnarray*}
\sup_{\stackrel{|u_1-u_2|<\delta}{0\le u_1<u_2 \le t}}
|\omega(u_1)-\omega(u_2)| <\frac{\down w_{\0}(\lt)-\up w_{\0}(\lt)}{2}.
\end{eqnarray*}
This implies that $\down t_k\geq s_{k}+\delta \geq \up t_k+\delta\ge \down t_{k-1}+\delta$ if $\down t_k < t$. 
So $\down t_{\cl{\frac{t}{\delta}}}\ge t$.

Let $(Y^{\uparrow*},Y^{\downarrow*})$ be the point-wise limit of the sequence $\{(\up Y_k,\down Y_k): k=1,2,\cdots\}$.
Since convergence can be achieved in a finite number of steps for any given $t$,  $(Y^{\uparrow*},Y^{\downarrow*})$ are finite at all $t\ge 0$. 
Taking the limit on both sides of \eqref{eq:ite_up} and \eqref{eq:ite_down}, by the Lipschitz continuity of $\up \psi(\arp,\down Y,\omega)$ and $\down \psi(\arp,\up Y,\omega)$, we can show that $(Y^{\uparrow*},Y^{\downarrow*})$ jointly satisfy \eqref{ite:down-to-up} and \eqref{ite:up-to-down}.

\noindent\underline{The uniqueness of the two-sided mapping:}
Finally, we prove the uniqueness of the two side mapping. 
Suppose that there exists a two-sided mapping $(Y^{\uparrow'},Y^{\downarrow'})$ that satisfies \eqref{ite:down-to-up} and \eqref{ite:up-to-down}.
By part~1 of Proposition~\ref{prop:order and continuous on one-side},  
$Y^{\downarrow'}\succeq \0$ implies $Y^{\uparrow'} \succeq \up Y_1$ and $Y^{\downarrow'} \succeq \down Y_1$, and subsequently, $Y^{\uparrow'} \succeq \up Y_i$ and $Y^{\downarrow'} \succeq \down Y_i$ for $i=2,3,\ldots$.
Thus, $Y^{\uparrow'} \geq Y^{\uparrow*}$ and $Y^{\downarrow'} \geq Y^{\downarrow*}$.
Define $\up\tau=\inf\{t\ge 0:Y^{\uparrow'}(t)>Y^{\uparrow*}(t)\}$ and $\down \tau=\inf\{t\ge 0:Y^{\downarrow'}(t)>Y^{\downarrow*}(t)\}$.
\begin{enumerate}
\item 
If $\up \tau > \down \tau$, then $Y^{\downarrow'}(u)=Y^{\downarrow*}(u)$ for $\down \tau \leq u<\up\tau$. 
Let
\begin{equation*}
  Y^{\uparrow''}(u)=\left\{ 
 \begin{array} {ll}
  \up Y(u), & u \in [0,\up\tau),\\
  Y^{\uparrow'}(u), & \textrm{otherwise}.
  \end{array}
  \right.
\end{equation*} 
Since $Y^{\uparrow''}(u)=\up Y(u)$ for $u<\up\tau$, $Y^{\uparrow''}$  is increasing and strictly less than $Y^{\uparrow'}$. 
By part 2 of Corollary~\ref{lemma:apply-der-to-H-L}, $Y^{\uparrow''} \in \up \Pi(\arp,Y^{\downarrow'},\omega)$, a contradiction.
\item 
If $\down \tau > \up \tau$, the proof is similar and omitted. 
\item 
If $\up \tau = \down \tau$, there exists some $\delta'>0$ such that both $Y^{\uparrow'}-\up Y$ and $Y^{\downarrow'}-\down Y$ are strictly positive in $(\up\tau,\up\tau+\delta)$. 
Denote $A_0=\arp(\up\tau+\lt)+\omega(\up\tau)+Y^{\uparrow'}(\up\tau)-Y^{\downarrow'}(\up\tau)$.
\begin{itemize}
\item  $A_0 \ge \frac{\down w_{\0}(\lt)+\up w_{\0}(\lt)}{2}$: Since $\arp(t)$, $\omega(t)$ and $Y^{\downarrow'}$ are right-continuous, there exists a $\delta\leq\delta'$ such that $|\arp(t+\lt)-\arp(\up\tau+\lt)|<\frac{\down w_{\0}(\lt)-\up w_{\0}(\lt)}{8}$, $|Y^{\downarrow'}(t)-Y^{\downarrow'}(\up\tau)|<\frac{\down w_{\0}(\lt)-\up w_{\0}(\lt)}{8}$ and  $|\omega(t)-\omega(\up\tau)|<\frac{\down w_{\0}(\lt)-\up w_{\0}(\lt)}{8}$ when $t\in [\up\tau,\up\tau+\delta)$.

Choose $\epsilon<\frac{\down w_{\0}(\lt)-\up w_{\0}(\lt)}{8}$ and let
\begin{equation*}
 Y^{\uparrow''}(u)=\left\{ 
 \begin{array} {ll} 
  \left\{ 
   \begin{array} {ll}
   Y^{\uparrow'}(\up\tau)-\epsilon  
    & \textrm { if } Y^{\uparrow'}(\up\tau) > \up Y(\up\tau),\\
   Y^{\uparrow'}(\up\tau) 
    & \textrm { if } Y^{\uparrow'}(\up\tau)= \up Y(\up\tau),
   \end{array}
  \right.     & u \in   [\up\tau,\up\tau+\delta),\\
  Y^{\uparrow'} &  \textrm {otherwise. }
  \end{array}
  \right. 
\end{equation*}
Then, under adjustments $(Y^{\uparrow''}, Y^{\downarrow'})$,
   \begin{eqnarray*}
  \arp_{t}(\lt)&=&\arp(t+\lt)+\omega(t)+Y^{\uparrow''}(t)-Y^{\downarrow'}(t)\\
              &\ge &\arp(\up\tau+\lt)+\omega(\up\tau)+Y^{\uparrow'}(\up\tau)-Y^{\downarrow'}(\up\tau)-\frac{\down w_{\0}(\lt)-\up w_{\0}(\lt)}{2} \\
              &\ge & A_0-\frac{\down w_{\0}(\lt)-\up w_{\0}(\lt)}{2} > \up w_{\0}(\lt)
\end{eqnarray*}
for $t \in [\up\tau,\up\tau+\delta)$. 
By Proposition~\ref{prop:two-apart-bounds}, we know $\arp_{t} \in \up {\bar{\mathbb D}}$ for $t \in [\up\tau,\up\tau+\delta)$. 
For $t \notin [\up\tau,\up\tau+\delta)$, $\arp_{t}\in \up {\bar{\mathbb D}}$ following the same argument as that in the proof of Proposition~\ref{prop:existence of one-side reflection mapping}. 
So $Y^{\uparrow''} \in \up \Pi(\arp,Y^{\downarrow'},\omega)$, a contradiction. 

\item $A_0<\frac{\down w_{\0}(\lt)+\up w_{\0}(\lt)}{2}$: Similarly, by finding the corresponding $\delta,\epsilon$ and letting
\begin{equation*}
 Y^{\downarrow''}(u)=\left\{ 
 \begin{array} {ll} 
  \left\{ 
   \begin{array} {ll}
   Y^{\downarrow'}(\up\tau)-\epsilon  
    & \textrm { if } Y^{\downarrow'}(\up\tau) > \down Y(\up\tau),\\
   Y^{\downarrow'}(\up\tau) 
    & \textrm { if } Y^{\downarrow'}(\up\tau)= \down Y(\up\tau),
   \end{array}
  \right.     & u \in   [\up\tau,\up\tau+\delta),\\
  Y^{\downarrow'}(u) &  \textrm {otherwise, }
  \end{array}
  \right.
\end{equation*}
we can show $Y^{\downarrow''} \in \down \Pi(\arp,Y^{\uparrow'},\omega)$, a contradiction. 
\end{itemize}
\end{enumerate}
\end{proof}

\subsection{The Optimality of the Two-sided Reflection Policy}
\label{sec:Optimality of the Two-Side Reflection Policy}
In this section, we show that the two-sided reflection mapping $\pi^*=(Y^{\uparrow *},Y^{\downarrow *})$ is  optimal and makes  the minimum amount
of adjustment to prevent the state $\arp_t$, $t \geq 0$, from falling into $ \up{\mathbb
D}$ and $\down{\mathbb D}$. 
Under the one-dimensional setting in \cite{HarrisonTaksar1983} and described
in Section $5$, $\up{\mathbb D}=\{y<b\}$ and $\down{\mathbb D}=\{y>a\}$,
 our two-sided reflection mapping reduces to the same closed-forms
\begin{eqnarray*}
 && R(t)=\sup\limits_{0 \le u \le t}{[a-\omega(u)+L(u)]},\ t \ge 0,\\
 && L(t)=\sup\limits_{0 \le u \le t}{[\omega(u)+R(u)-b]},\ t \ge 0
\end{eqnarray*} in their paper. 
This reflection mapping makes the minimum amount of adjustment to keep the controlled process in the region $\{a \le y \le b\}$.
\begin{theorem}
\label{th:optimal policy}
The policy $\pi^*=(Y^{\uparrow*},Y^{\downarrow*})$ is optimal, i.e.,
$ C(\arp,\pi^*)=C^*(\arp) \mbox{ for all } \arp \in \mathbb{D}$.
\end{theorem}

We prove Theorem \ref{th:optimal policy} by considering a cost characterized by $\delta>0$ in \eqref{eq:C-mod} and showing that this cost approaches both $C(\arp,\pi^*)$ (Lemma \ref{lemma:close-1}) and $C^*(\arp)$ (Lemma \ref{lemma:close-2}) as $\delta \to 0$.  
Let $\arp_t$ be the state at $t$ under the two-sided reflection policy $\pi^*=(Y^{\uparrow*},Y^{\downarrow*})$ with initial profile $\arp$. 
For any small $\delta>0$, let $\down{\mathbb D}-\delta$ = $\{\arp'-\delta,:\forall\arp' \in \down{\mathbb D}\}$ and $\up{\mathbb D}+\delta:=\{\arp'+\delta:\forall\arp' \in \up{\mathbb D}\}$.
For a given sample path of the Brownian motion and associated control $\pi^*$, the state $\arp_t$ will enter $\up{\mathbb D}+\delta$ when $\frac{\partial C(\arp_t-\delta,0,0)}{\partial\up \xi}=0$ and $\down{\mathbb D}-\delta$ when $\frac{\partial C(\arp_t+\delta,0,0)}{\partial\down \xi}=0$ many times over time.    
Without loss of generality, we assume that $\arp_t$ first enters $\up{\mathbb D}+\delta$ and at 
\begin{equation*}
  \tau^{\delta}_1 = \inf\left\{t \geq 0: \frac{\partial C(\arp_t-\delta,0,0)}{\partial\up
\xi}=0 \right\}.
\end{equation*}
The process evolves and eventually enters $\down {\mathbb D}-\delta$ at 
\begin{equation*}
  \tau^{\delta}_2 = \inf\left\{t > \tau^{\delta}_1: \frac{\partial C(\arp_t+\delta,0,0)}{\partial\down
\xi}=0 \right\}.
\end{equation*}
For $j=1,2,\cdots$, define 
\begin{align*}
  \tau^{\delta}_{2j+1} 
  &= \inf\left\{t >\tau^{\delta}_{2j}: 
    \frac{\partial C(\arp_t-\delta,0,0)}{\partial\up \xi}=0 \right\},\\
  \tau^{\delta}_{2j+2} 
  &= \inf\left\{t  >\tau^{\delta}_{2j+1}:
    \frac{\partial C(\arp_t+\delta,0,0)}{\partial\down \xi}=0 \right\}.
\end{align*}
 $\tau^{\delta}_{2j+1}$ represents the first time $\arp_t$ enters $\up{\mathbb D}+\delta$ since $\tau^{\delta}_{2j}$, and $\tau^{\delta}_{2j+2}$ represents the first time $\arp_t$ enters $\down{\mathbb D}-\delta$ since $\tau^{\delta}_{2j+1}$.
Thus, $\tau^{\delta}_i$, $i=1,2, \cdots,$ form a series of stopping times. 
Let $N(t)=\max\{k:\tau^{\delta}_k \le t\}$ be the total number of such stopping times by $t$,
\begin{eqnarray*}
  \arp^{\delta}_t = \left\{
  \begin{array}{ll}
    \arp_t-\delta, & \textrm{ if }\ t< \tau^{\delta}_1,\\
    \arp_t+\delta, & \textrm{ if }\ \tau^{\delta}_{2j-1}\leq t <\tau^{\delta}_{2j},\\
    \arp_t-\delta, & \textrm{ if }\ \tau^{\delta}_{2j}\leq t <\tau^{\delta}_{2j+1},
  \end{array}
  \right.
\end{eqnarray*}
and 
\begin{equation}
  \label{eq:C-mod}
  C^{\delta}(\arp,\pi^*) 
  = \E\left[
    \int_0^{\infty} e^{-\gamma t}h(\arp^{\delta}_t(0))dt
    + \up k \int_0^{\infty} e^{-\gamma t} d Y^{\uparrow*}(t)
    + \down k\int_0^{\infty}e^{-\gamma t} d Y^{\downarrow*}(t)
  \right]
\end{equation}
be the cost associated with the process $\{\arp^\delta_t\}$ and policy $\pi^*$. 
$C^{\delta}(\arp,\pi^*)$ differs from $C(\arp,\pi^*)$ only by the holding cost term and the difference is bounded by $\int_0^{\infty} e^{-\gamma t} t dt=\frac{M}{\gamma}\delta$ as stated in the following lemma.

\begin{lemma}
  \label{lemma:close-1}
$|C(\arp,\pi^*)-C^{\delta}(\arp,\pi^*)| \leq \frac{M}{\gamma}\delta.$
\end{lemma}

Applying Proposition~\ref{prop:characterization of the mapping} and Theorem~\ref{thm:PDE}, we can show the following lemma.
\begin{lemma}
  \label{lemma:close-2}
For any fixed $T \ge 0$,
\begin{equation}
 C^{\delta}(\arp,\pi^*) \leq C^*(\arp)+(2\E N(T)+3)M\delta-R_1(\arp,\delta,T)+R_2(T),
\label{eq:delta-change-cost}
  \end{equation}
where $R_1(\arp,\delta,T) \to 0$ as $\delta \to 0$ for any fixed $T$ and $R_2(T) \to 0$ as $T \to \infty$.
\end{lemma}

The proof is quite technical and can be found in the Appendix. 
We are now ready to prove the optimality of the two-sided reflection policy $\pi^*$.
\begin{proof}[Proof of Theorem~\ref{th:optimal policy}]
We first show that $\E N(t)$ is finite for any $t \geq 0$.
Consider a sequence of stopping times of the Brownian motion $W_t,$ 
\begin{eqnarray*}
&&U_1=\inf\left\{t>0,  |W_t|=\frac{\uub-\ulb}{4}\right\},\\
&&U_j=\inf\left\{t>U_{j-1}, |W_t-W_{U_{j-1}}|=\frac{\uub-\ulb}{4}\right\},\quad j=1,2,\cdots.
\end{eqnarray*}
and let $N'(t)=\max\{j:U_j\le t\}$ be the corresponding counting process.  

By the definitions of two consecutive stopping times $\tau^{\delta}_{2j-1}$ and $\tau^{\delta}_{2j}$, $\arp^{\delta}_t$ enters $\down{\mathbb D}-\delta$ at $\tau^{\delta}_{2j-1}$ and then enters $\up{\mathbb D}+\delta$ at $\tau^{\delta}_{2j}$. 
By the same argument leading to \eqref{eq:tech-fluc-W} in the proof of Proposition~\ref{prop:existence of two-side reflection mapping}, for a small enough $\delta$, there exist ${\tau^{\delta}_{2j-1}\le s_1<s_2\le\tau^{\delta}_{2j}}$ such that
\begin{equation*}
  W_{s_1}-W_{s_2}
  \ge \uub-\ulb-2\delta
  > \frac{\uub-\ulb}{2}.
\end{equation*}
Thus, there must exist $i_j$ such that $U_{i_j} \in [s_1,s_2] \subset [\tau_{2j-1}^{\delta},\tau_{2j}^{\delta}]$ for each $j=1,2,\cdots$. 
Hence $N(t)\leq 2N'(t)$ for any $t>0$ and $\E N(t)$ is finite.
Fixing the $T$ and letting $\delta \to 0$ in Lemma~\ref{lemma:close-1} and Lemma~\ref{lemma:close-2}, we have
\begin{equation}
  C(\arp,\pi^*) \leq C^*(\arp) + R_2(T)
 \label{eq:final}
\end{equation} 
for any  $T \ge 0$.
Note that $R_2(T) \to 0$ as $T \to \infty$, combining the above with the optimality of $C^*(\arp)$, we have $C(\arp,\pi^*)=C^*(\arp)$.
\end{proof}

\section{Conclusions and Discussion of the General Case with a Lead Time for Downward Adjustments}
\label{sec:discussion}
In this paper, we consider the optimal control of a storage system whose content is driven  by a Brownian motion absent  control.  
Because there is a positive lead time for upward adjustments, the state of the system is a function on a continuous interval and  such a problem is extremely challenging. 
We develop a novel four-step approach described in the Introduction to identify the structure of optimal control as a state-dependent two-sided reflection mapping that makes the minimum amount of upward or downward adjustment to prevent the state from entering into certain regions.  
To the best of our knowledge, this is the first paper to study instantaneous control of stochastic systems in a functional setting and the methodology developed in the paper may inspire ways to solve other control problems in various applications.  

We have assumed that downward adjustments are instantaneous. If they are not and there is  a positive lead time for downward adjustments, then by the time a promised downward adjustment is made there may not be enough content left due to the Brownian motion. 
The only way to avoid this situation completely is to add a constraint on downward adjustments and set aside enough inventory.
But then it will be too difficult to calculate the inventory cost. 

Now suppose that backlogging of downward adjustments after the lead time is allowed at the same penalty cost as that whenever the content is negative. 
Then, if the lead times for upward and downward adjustments are identical, the problem can be reduced to one with zero upward and downward adjustment lead times by Theorem~3.11 in \cite{OksendalSulem2009}.
Otherwise, our analysis can be extended by transforming the problem into one with a single lead time as follows. 

Since upward and downward adjustments are symmetric analytically when the latter can be backlogged, we only need to consider the case where $\up\lt \geq \down\lt>0$ and show that the system can be transformed into one with zero lead time for downward adjustments. 

Define ${\arp}^{i}_t(u)$ as the total outstanding movement $i$, $i \in \{\uparrow,\downarrow \}$, at time $t$ but {\it before} any adjustment at time $t$ that will be realized during $(t,t+u]$ and $\arp^{i}_t=\{\arp^{i}_t(u), u \geq 0\}$.
Then, $(\up{\arp}_t,\down{\arp}_t)$ is the profile of the outstanding movements at time $t$ with $\arp^{i}_t(0)=0$ and $\arp^{i}_t(u)=\arp^{i}_t(\lt^{i})$ for $u>\lt^{i}$, and $(H_t, \up{\arp}_t, \down{\arp}_t)$ describes the state of the system at time $t$.
Hence, for $t>0$, the dynamics of the system can be written as\begin{align}
  \label{eq:dyn-value}
  H_t&= H_0 + W_t + \up{\arp}_{0}(t) - \down{\arp}_0(t)
      + \up{Y}(t-\up{\lt}) - \down{Y}(t-\down{\lt}),\\ \label{eq:dyn-profile}
  {\arp}^i_{t}(u) &= \left\{
   \begin{array}{ll}
  {\arp}^i_0(t+u)-{\arp}^i_0(t)+{Y}^i(t+u-\lt^i)-{Y}^i(t-\lt^i), 
   & \mbox{if   } u \leq \lt^i,\\
  {\arp}^i_{t}(\lt^i), 
   & \mbox{else,}
    \end{array} 
    \right.
\end{align}
and the cost function for any initial state $(H_{0},\up{\arp}_{0},\down{\arp}_{0})$ and policy $\pi$ is  
\begin{equation}
 \label{eq:cost_function}
  \tilde{C}(H_{0},\up{\arp}_{0},\down{\arp}_{0},\pi)
  =\E\left[
    \int_0^{\infty}e^{-\gamma t}h(H_t) dt 
     + \int_0^{\infty}e^{-\gamma t} \up k d \up Y(t)
     + \int_0^{\infty}e^{-\gamma t} \down k d \down Y(t)
  \right].
\end{equation}


Now consider another system where there is no lead time for downward adjustments and the lead time for upward adjustments is $\lt=\up\lt-\down\lt$,  the initial state is $\arp_0(u)=H_0+\up{\arp}_0(u+\down\lt)-\down{\arp}_0(\down\lt),$ and the holding cost rate is $\tilde{h}(x)=e^{-\gamma \down\lt}E[h(x+\mathcal{N}_{\down\lt})]$.
The following proposition reveals that the difference between the cost functions of the single lead time system and the original system is a constant under the same policy. 
Thus, the problem reduces to one with zero lead time for downward adjustments.
\begin{proposition}
\label{prop:costequivalence}
For any fixed policy $\pi$,   
 \begin{equation*}
    C(\arp_{0},\pi) 
    = \tilde{C}(H_{0},\up{\arp}_{0},\down{\arp}_{0},\pi)
    -\E\left[\int_0^{\down{\lt}}e^{-\gamma t}{h}(H_t) dt\right], 
\end{equation*}
where $\arp_0(u)=H_0+\up{\arp}_0(u+\down\lt)-\down{\arp}_0(\down\lt)$ and  $\E \left[\int_0^{\down{\lt}}e^{-\gamma t}h(H_t) dt \right]$ is a constant  for given  $(H_{0},\up{\arp}_{0},\down{\arp}_{0})$.
\end{proposition}

\bibliography{pub}

\newpage

\section*{Appendix}

\begin{proof}[Proof of Proposition~\ref{prop:Lips-con of C^*}]
By the definition of $C^*(\arp)$, for any $\epsilon>0$, we can find a policy $\pi$ such that
$C(\arp,\pi) \leq C^*(\arp)+\epsilon$.
We apply the same policy $\pi$ to the state $\arp'$ and denote $\arp_t$ and $\arp'_t$ to be the states under $\pi$ with initial state $\arp$ and $\arp'$, respectively.
\begin{eqnarray*}
 C^*(\arp')-C^*(\arp)-\epsilon & \leq &  C(\arp',\pi)-C(\arp,\pi) 
   = \E\left[
    \int_0^{\infty}e^{-\gamma t}[h(\arp'_t(0))-h(\arp_t(0))]dt
       \right]\\
    &\leq &M\int_0^{\infty}e^{-\gamma t}  [(\arp'(t)-\arp(t)]dt=Md(\arp,\arp').
    \end{eqnarray*}
By symmetry, we also have $C^*(\arp)-C^*(\arp')-\epsilon \leq M d(\arp',\arp)$.
Letting $\epsilon \to 0$, we have that $C^*(\arp)$ is Lipschitz continuous.
\end{proof}

\begin{proof}[Proof of Proposition~\ref{prop:midstep-convex}]
For any given state ${\mathbf x}=(x_0,x_1,\cdots,x_{n-1})$, if we let $x_n=x_{n-1}+\up y$, we can rewrite
\begin{eqnarray*}
c^{T,n}_t(\mathbf{x},\up y_,\down y)&=& \up k x_{n} +\down k \down y-\up k x_{n-1} \\
&&+\alpha \E\left[C^{T,n}_{t+1}((x_{1},x_{2},\cdots,x_{n-1},x_{n})-\down y\e+w_t \e)+h^n(x_{0}-\down y)\right] 
\end{eqnarray*} and view $c^{T,n}_t(\mathbf{x},\up y_,\down y)$ as a function of $(\mathbf{x}_,x_{n},\down y)$. 
We next show by induction that $c^{T,n}_t(\mathbf{x}_,\up y_,\down y)$ is $L^\natural$-convex in $(\mathbf{x}_,x_{n},\down y)$ and $C^{T,n}_t(\mathbf{x})$ is $L^\natural$-convex in $\mathbf{x}$ simultaneously.  

Since $h^n(x)$ is convex, $C^{T,n}_T(\mathbf{x})$ is $0$ and hence $L^\natural$-convex in $\mathbf{x}$.
Assuming that $C^{T,n}_{t+1}(\mathbf{x})$ is  $L^\natural$-convex in $\mathbf{x}$.
Since $h^n(\cdot)$ is convex and $C^{T,n}_{t+1}((x_{1},x_{2},\cdots,x_{n-1},x_{n})-\down y \e+w_t \e)$ is $L^\natural$-convex in $(x_1,\cdots,x_n,\down y)$ for a given $w_t$, by Lemma~1 in \cite{Zipkin2008}, $c^{T,n}_{t}(\mathbf{x},\up y,\down y)$ is $L^\natural$-convex in $(\mathbf{x},x_n,\down y)$ as $L^\natural$-convexity is preserved by expectation. 
Thus,
\begin{equation*}
C^{T,n}_t(\mathbf{x})
 =\min_{x_n \geq x_{n-1},\down y \ge 0} \left\{c^{T,n}_{t}(\mathbf{x},\up y,\down y)\right\}
 =\min_{x_n \geq x_{n-1}} 
  \left\{\min_{\down y \ge 0}
  \left\{c^{T,n}_{t}(\mathbf{x},\up y,\down y)\right\}
  \right\}
\end{equation*}
is $L^\natural$-convex in $\mathbf{x}$ by Lemma~2 in \cite{Zipkin2008} as minimization over a
sublattice preserves $L^\natural$-convexity. 
\end{proof}

\begin{proof}[Proof of Lemma~\ref{prop:mono-boundary-con}] \
\begin{enumerate}
\item Since $\phi(\up\xi, \down\xi)$ is a linear function of $(\up \xi,\down \xi)$, we only need to show the monotonicity of $\frac{\partial C^*(\Phi_{\up\xi,\down\xi}(\arp))}{\partial \up \xi}$.
For any $\epsilon >0$ and  $\arp_1 \preceq \arp_2$ where $\arp_1(\lt)=\arp_2(\lt)$, $\Phi_{\up\xi+\epsilon,\down\xi}(\arp_1)\vee \Phi_{\up\xi,\down\xi}(\arp_2)=\Phi_{\up\xi+\epsilon,\down\xi}(\arp_2)$ and $\Phi_{\up\xi+\epsilon,\down\xi}(\arp_1)\wedge \Phi_{\up\xi,\down\xi}(\arp_2)=\Phi_{\up\xi,\down\xi}(\arp_1)$.
Since $C^*(\arp)$ is $L^\natural$-convex, letting $\xi_1=\xi_2=0$ and $F=C^*$ in Definition~\ref{def:L-natural-convexity}, we have
  \begin{equation*}
    C^*(\Phi_{\up\xi+\epsilon,\down\xi}(\arp_1))+C^*(\Phi_{\up\xi,\down\xi}(\arp_2))
    \geq 
    C^*({\Phi_{\up\xi+\epsilon,\down\xi}(\arp_2)})+C^*({\Phi_{\up\xi,\down\xi}(\arp_1)}),
  \end{equation*}
 or
  \begin{equation*}
     C^*(\Phi_{\up\xi+\epsilon,\down\xi}(\arp_1))-C^*(\Phi_{\up\xi,\down\xi}(\arp_1))
     \ge 
     C^*({\Phi_{\up\xi+\epsilon,\down\xi}(\arp_2)})-C^*({\Phi_{\up\xi,\down\xi}(\arp_2)}),
  \end{equation*}
  which implies the monotonicity of  $\frac{\partial C^*(\Phi_{\up\xi,\down\xi}(\arp))}{\partial \up \xi}$.

\item
For any $\epsilon,a >0$, letting $F=C^*$, $\arp_1=\Phi_{\up\xi,0}(\arp)$, $\arp_2=\Phi_{\up\xi+\epsilon,0}(\arp)$ and $(\xi_1,\xi_2)=(0,-a)$ in Definition~\ref{def:L-natural-convexity}, we have
 \begin{equation*}
    C^*(\Phi_{\up\xi,\down \xi}(\arp)-0)+C^*(\Phi_{\up\xi+\epsilon,\down \xi}(\arp)-(-a))\geq 
    C^*(\Phi_{\up\xi,\down \xi}(\arp)-(-a))+C^*(\Phi_{\up\xi+\epsilon,\down \xi}(\arp)-0),
  \end{equation*}
which implies $\frac{\partial C^*(\Phi_{\up\xi,\down\xi}(\arp)+a)}{\partial \up \xi} \ge  
\frac{\partial C^*(\Phi_{\up\xi,\down\xi}(\arp))}{\partial \up \xi}$and  the result holds.

\item
For any $\epsilon >0$ and $\arp_1 \succeq \arp_2$, letting $F=C^*$ and $(\xi_1,\xi_2)=(\xi,\xi+\epsilon)$ in Definition~\ref{def:L-natural-convexity}, we have
 \begin{equation*}
    C^*(\arp_1-\xi)+C^*(\arp_2-(\xi+\epsilon))\geq 
    C^*(\arp_2-\xi)+C^*(\arp_1-(\xi+\epsilon)),
  \end{equation*}
which implies  $\frac{\partial C^*(\Phi_{0,\down\xi}(\arp_1))}{\partial \down \xi} \ge  
\frac{\partial C^*(\Phi_{0,\down\xi}(\arp_2))}{\partial \down \xi}$.
Replacing $\arp_1$ and $\arp_2$ by $\Phi_{\up \xi,0}(\arp_1)$ and $\Phi_{\up \xi,0}(\arp_2)$, we have that $\frac{\partial C(\arp,\up \xi,\down \xi)}{\partial\down \xi}$ is increasing in $\arp$.
\end{enumerate}
\end{proof}

\begin{proof}[Proof of equations \eqref{equ:A happens} and \eqref{equ:A^c happens}]
Note that, under the periodic policy $\pi^n$, adjustments can only be  made at $T^n_i=\frac{i\lt}{n}$ and at the amounts $(\xi_i^{n\uparrow},\xi_i^{n\downarrow})$ for $i=0,1,2,\cdots$.
For convenience, we use $T_i$ to represent $T^n_i$ and $(\xi_i^{\uparrow},\xi_i^{\downarrow})$ to represent $(\xi_i^{n\uparrow},\xi_i^{n\downarrow})$ for $i=0,1,2,\cdots$ in this proof.
\begin{itemize}
\item 
On the event $A$, rewrite  $C(\arp',\pi^n)$ as \begin{equation*}
    \E\bigg[
      \int_0^{T_{N(\epsilon)}} e^{-\gamma t}h(\arp'_t(0)) dt 
      +\sum_{i=1}^{N(\epsilon)}e^{-\gamma T_i}\phi(\xi^{\uparrow}_i,\xi^{\downarrow}_i)
    \bigg] 
    +  \E\bigg[
      \int_{T_{N(\epsilon)}}^{\infty} e^{-\gamma t}h(\arp'_t(0)) dt 
      +\sum_{N(\epsilon)+1}^{\infty}e^{-\gamma T_i}\phi(\xi^{\uparrow}_i,\xi^{\downarrow}_i)
    \bigg],
  \end{equation*}
where the second item is  the discounted control cost given initial state $\arp'_{T_{N(\epsilon)}}$ and is thus always larger than or equal to the lower bound $\E\left[e^{-\gamma T_{N(\epsilon)}} C^* \left(\arp'_{T_{N(\epsilon)}}\right)\right]$.
Hence, we have
\begin{align}
   C(\arp',\pi^n)  \geq  \E\bigg[
      \int_0^{T_{N(\epsilon)}} e^{-\gamma t}h(\arp'_t(0)) dt 
      +\sum_{i=1}^{N(\epsilon)}e^{-\gamma T_i}\phi(\xi^{\uparrow}_i,\xi^{\downarrow}_i)
    \bigg] 
     + \E\left[e^{-\gamma T_{N(\epsilon)}} C^*(\arp'_{T_{N(\epsilon)}})\right].
\label{equ:two-part-C}
\end{align}
Note that the optimal cost $C^*(\arp')$  can be written as
\begin{equation}
\label{equ:h-C}
\begin{split}
  C^*(\arp')
  &=  \E\bigg[
      \sum_{i=1}^{N(\epsilon)}[e^{-\gamma T_{i-1}}C^*(\arp'_{T_{i-1}})-e^{-\gamma T_i}C^*(\arp'_{T_i-})]
      \bigg] \\
  &\quad +\E\bigg[
      \sum_{i=0}^{N(\epsilon)}e^{-\gamma T_i}[C^*(\arp'_{T_{i}-})-C^*(\arp'_{T_i})]
      \bigg]
      +\E\left[e^{-\gamma T_{N(\epsilon)}} C^*(\arp'_{T_{N(\epsilon)}})\right].
\end{split}      
\end{equation}
By the optimality condition \eqref{optimality2}, the first term in \eqref{equ:h-C}
is smaller than
\begin{equation}
  \label{eq:C*-h}
  \sum\limits_{i=1}^{N(\epsilon)}\E\left[\int_{T_{i-1}}^{T_{i}} e^{-\gamma t}h(\arp_t'(0)+W_t)
dt\right]
  = \E\left[\int_0^{T_{N(\epsilon)}} e^{-\gamma t}h(\arp_t'(0)+W_t) dt\right].
\end{equation}
By the dynamics \eqref{equ:dynamics} and the definition of $C(\arp,\up \xi,\down
\xi)$, we have $C^*(\arp'_{T_i})=C^*(\Phi_{\xi^{\uparrow}_i,\xi^{\downarrow}_i}(
\arp'_{T_i-}))=C(\arp'_{T_i-},\xi^{\uparrow}_i,\xi^{\downarrow}_i)-
\phi(\xi^{\uparrow}_i,\xi^{\downarrow}_i)$ and the second term in 
\eqref{equ:h-C} can be written as
\begin{equation}
  \label{eq:tech-510-2}
  \E\left[
    \sum_{i=0}^{N(\epsilon)}e^{-\gamma T_i}[C(\arp'_{T_{i}-},0,0)-C(\arp'_{T_i-},\xi^{\uparrow}_i,\xi^{\downarrow}_i)]\right]
  +\E\left[
    \sum_{i=0}^{N(\epsilon)}e^{-\gamma T_i}\phi(\xi^{\uparrow}_i,\xi^{\downarrow}_i)
  \right].
\end{equation}
Let $A$ denote the event where $\{T_{N(\epsilon)} \leq \tau'\}$ and $A^c$ its complement, and $\E_A[X]=\E[X\id{A}]$ for any random variable $X$. 
Then \eqref{eq:tech-510-2} can be bounded from above by 
\begin{equation}
   \E_A\left[
    \sum_{i=0}^{N(\epsilon)}e^{-\gamma T_i}[C(\arp'_{T_{i}-},0,0)-C(\arp'_{T_i-},\xi^{\uparrow}_i,\xi^{\downarrow}_i)]\right] 
  +\E\left[
    \sum_{i=0}^{N(\epsilon)}e^{-\gamma T_i}\phi(\xi^{\uparrow}_i,\xi^{\downarrow}_i)
  \right], 
  \label{4}
\end{equation} after dropping the term $\E_{A^c}[\cdot]$.
Since $C(\arp,0,0)-C(\arp,\up \xi,\down \xi)$ is always non-positive for any $(\arp,\up \xi,\down \xi)$ by Proposition~\ref{prop:mono-boundary-con} and $C(\arp,\up \xi,\down \xi)$ is convex in $\up \xi $ and $\down \xi$, respectively, 
\begin{eqnarray}
 & & C(\arp'_{T_{i}-},0,0)-C(\arp'_{T_i-},\xi^{\uparrow}_i,\xi^{\downarrow}_i)\nonumber\\
  &=&\left[C(\arp'_{T_i-},0,0)-C(\arp'_{T_i-},\xi^{\uparrow}_i,0)\right]
+\left[C(\arp'_{T_{i}-},\xi^{\uparrow}_i,0)-C(\arp'_{T_i-},\xi^{\uparrow}_i,\xi^{\downarrow}_i)\right]\nonumber\\
  &\leq& -\frac{\partial C(\arp'_{T_i-},\xi^{\uparrow}_i,0)}{\partial \up\ja} \xi^{\uparrow}_i
         -\frac{\partial C(\arp'_{T_i-},\xi^{\uparrow}_i,\xi^{\downarrow}_i)}{\partial \down\ja} \xi^{\downarrow}_i\nonumber\\
  &=& -\frac{\partial C(\Phi_{\xi^{\uparrow}_i,0}(\arp'_{T_i-}),0,0)}{\partial \up\ja} \xi^{\uparrow}_i
-\frac{\partial C(\arp'_{T_i},0,0)}{\partial \down\ja} \xi^{\downarrow}_i.
\label{eq:tech-derivative-k0}
\end{eqnarray}

On the event $A$, for any $k \leq N(\epsilon)$, $(W_{T_i},T_i)$ is in the set $(\hat{w}-w,\hat{s}-s)+B(\delta)$.  
Consequently, $(w+W_{T_i},s+T_i)$ is in $(\hat{w},\hat{s})+B(\delta)$. 
Moreover, the cumulative amount of upward  and downward adjustments at time $T_i$ is less than $\delta$, which means $\sum\limits_{i \leq k} \xi^{\uparrow}_i +\sum\limits_{i \leq k}\xi^{\downarrow}_i\le \epsilon \le \delta$. 
By \eqref{equ:dynamics},
{\small  
\begin{align*}
  d(\ts_{\hat{s}}(\arp)+\hat{w},\arp'_{T_i})
  &\leq d(\ts_{\hat{s}}(\arp)+\hat{w},\ts_{(s+T_i)}(\arp)+w+W_t) + \sum\limits_{i
\leq k}\xi^{\uparrow}_i +\sum\limits_{i \leq k}\xi^{\downarrow}_i\\
  &\leq d(\ts_{\hat{s}}(\arp)+\hat{w},\ts_{(s+T_i)}(\arp)+\hat{w}) +    d(\ts_{(s+T_i)}(\arp)+\hat{w},\ts_{(s+T_i)}(\arp)+w+W_t)
+ \delta\\
  &\leq (s+T_i-\hat{s}) \gamma \arp(\lt) +2\delta \leq 3\delta.
\end{align*}}
Similarly, we have $d(\ts_{\hat{s}}(\arp)+\hat{w},\Phi_{\xi^{\uparrow}_i,0}(\arp'_{T_i-}))\le 3\delta$.
Thus, by \eqref{eq:subregion}, we have 
$\frac{\partial C(\Phi_{\xi^{\uparrow}_i,0}(\arp'_{T_i-}),0,0)}{\partial \up\ja} \geq k_0$ and $\frac{\partial C(\arp'_{T_i},0,0)}{\partial \down\ja} \geq k_0$.  
That it, \eqref{eq:tech-derivative-k0} is bounded by $-k_0(\xi^{\uparrow}_i+\xi^{\downarrow}_i)$ on the event $A$.
Consequently, the first term in \eqref{4} is bounded from above by
\begin{equation}
  \label{eq:tail}
    \E_A\bigg[
    \sum_{i=0}^{N(\epsilon)}-e^{-\gamma \delta}k_0 (\up \xi_i+\down \xi_i)
  \bigg]
  \leq  -\prob(A) e^{-\gamma \delta} k_0 \epsilon.
\end{equation}
Plugging \eqref{eq:C*-h}, \eqref{4} and \eqref{eq:tail} into \eqref{equ:h-C}, we have 
\begin{equation*}
   C^*(\arp')  \leq  \E\bigg[
      \int_0^{T_{N(\epsilon)}} e^{-\gamma t}h(\arp'_t(0)) dt 
      +\sum_{i=1}^{N(\epsilon)}e^{-\gamma T_i}\phi(\xi^{\uparrow}_i,\xi^{\downarrow}_i)
      +e^{-\gamma T_{N(\epsilon)}} C^*(\arp'_{T_{N(\epsilon)}})
    \bigg] 
    -\prob(A) e^{-\gamma \delta} k_0 \epsilon.
\end{equation*}
Comparing it with \eqref{equ:two-part-C}, we have
\begin{equation*}
  C(\arp',\pi^n) \geq C^*(\arp') +\prob(A) e^{-\gamma \delta} k_0 \epsilon
              \geq V_{\arp'}(0,0)+ \prob(A) e^{-\gamma \delta} k_0 \epsilon.
\end{equation*}

\item On the event $A^c$, rewrite $C(\arp',\pi^n)$ as
  \begin{equation*}
     \E\bigg[
      \int_0^{\tau'} e^{-\gamma t}h(\arp'_t(0)) dt 
      +\sum_{T_i\leq \tau'}e^{-\gamma T_i}\phi(\xi^{\uparrow}_i,\xi^{\downarrow}_i)
    \bigg] 
    + \E\bigg[
      \int_{\tau'}^{\infty} e^{-\gamma t}h(\arp'_t(0)) dt 
      +\sum_{T_i>\tau'} e^{-\gamma T_i}\phi(\xi^{\uparrow}_i,\xi^{\downarrow}_i)
    \bigg].
  \end{equation*}
Similar to the argument in \eqref{equ:two-part-C}, the second term is greater than $\E\left[e^{-\gamma \tau'} C^*(\arp'_{\tau'})\right]$ . 
Dropping the non-negative item $\E_{A}[\cdot]$ in the expectations, we have
{\small
\begin{align}
  C(\arp',\pi) 
  &\geq  \E_{A^c}\bigg[
          \int_0^{\tau'} e^{-\gamma t}h(\arp'_t(0)) dt 
          +\sum_{T_i\leq \tau'}e^{-\gamma T_i}\phi(\xi^{\uparrow}_i,\xi^{\downarrow}_i)
          \bigg] 
    + \E_{A^c}\left[e^{-\gamma \tau'} C^*(\arp'_{\tau'})\right] \nonumber\\
  &\geq \E_{A^c}\bigg[
      \int_0^{\tau'} e^{-\gamma t}h(\ts_t(\arp')(0)+W_t) dt 
    \bigg] 
  + \E_{A^c}\left[e^{-\gamma \tau'} C^*(\ts_{\tau'}(\arp')+W_{\tau'})\right]- \frac{M \epsilon}{\gamma}.  
  \label{equ:A^C}
\end{align}}The second inequality follows because, on the event $A^c$, the cumulative amount of upward and downward adjustments by the stopping time $\tau'$ is less than $\epsilon$.
Thus, by \eqref{equ:dynamics}, the distance $d(\ts_s(\arp')+W_s,\arp'_s)<\epsilon$ for any $0 \leq s \leq \tau'$.  
By Assumption~\ref{assum:orig-h}, $|h(\arp'_s(0))-h(\ts_s(\arp')(0)+W_s)|
\leq M \epsilon$ and by Proposition~\ref{prop:Lips-con of C^*}, $|C^*(\ts_s(\arp')+W_s)-C^*(\arp'_s)|
\leq \frac{M}{\gamma} d(\ts_s(\arp')+W_s,\arp'_s) < \frac{M}{\gamma} \epsilon$ for any $0 \leq s \leq \tau'$.
For each of the expectation $\E_{A^c}[\cdot]$ in (\ref{equ:A^C}), we can write it as the difference $\E[\cdot]-\E_A[\cdot]$.
Since the process $(W_t,t)$ doesn't go out of $(\hat{w}-w,\hat{s}-s)+B(\delta) $ before the stopping time $\tau'$, the shifted process $(w+W_t,s+t)$ is always in $(\hat{w},\hat{s})+B(\delta)$ for all $0 \le t \le \tau'$.
Then, for the $\E_A[\cdot]$ terms, we have the following bound
\begin{eqnarray}
&&  \E_{A}\left[
      \int_0^{\tau'} e^{-\gamma t}h(\ts_t(\arp')(0)+W_t) dt 
    \right] 
 +\E_{A}\left[e^{-\gamma \tau'} C^*(\ts_{\tau'}(\arp')+W_{\tau'})\right]\nonumber\\
&=& \E_{A}\left[
      \int_0^{\tau'} e^{-\gamma t}h(\ts_{s+t}(\arp)(0)+w+W_t) dt 
    \right] 
 +\E_{A}\left[e^{-\gamma \tau'} V_{\arp}(w+W_{\tau'},s+\tau')\right]\nonumber\\
&\leq& \prob(A) \int_0^{\delta} \bar{h} dt +\prob(A) \bar{V} =\prob(A) (\delta\bar{h}+
\bar{V}),\label{equ:P-A^c-term}
\end{eqnarray}
where $\bar{h}=\sup\limits_{(w,s)\in (\hat{w},\hat{s})+B(\delta)}\{h(\ts_{s}(\arp)(0)+w)\}<\infty$ and $\bar{V}=\sup\limits_{(w,s)\in (\hat{w},\hat{s})+B(\delta)}\{V_{\arp}(w,s)\}<\infty$, all independent of $\tau'$.
This means that $C^*(\ts_{\tau'}(\arp')+W_{\tau'})=V_{\arp}(w+W_{\tau'},s+\tau')
\leq \bar{V}$ and $h(\ts_t(\arp')(0)+W_t)=h(\ts_{s+t}(\arp')(0)+w+W_t)
\leq \bar{h}$ for all $t \leq \tau'$.
Plugging  \eqref{equ:P-A^c-term} into \eqref{equ:A^C}, we have
\begin{eqnarray*}
C(\arp',\pi) &\geq&  \E\left[
      \int_0^{\tau'} e^{-\gamma t}h(\ts_t(\arp')(0)+W_t) dt 
    \right] + \E\left[e^{-\gamma \tau'} C^*(\ts_{\tau'}(\arp')+W_{\tau'})\right]\\
   &&-\frac{M\epsilon}{\gamma}-\prob(A) (\delta\bar{h}+ \bar{V}).
\end{eqnarray*}
Comparing the above with \eqref{eq:o-inequ-small-neighbour}, we have
\begin{eqnarray*}
  C(\arp',\pi) 
  \geq
  V_{\arp'}(0,0) +c_0-\frac{M \epsilon}{\gamma}-\prob(A) (\delta\bar{h}+
\bar{V}).
\end{eqnarray*}
\end{itemize}
\end{proof}

\begin{proof}[Proof of Proposition~\ref{prop:PDE-positive}]
  Since $C^*(\arp)$ is $L^\natural$-convex, the partial derivatives $\frac{\partial V_{\arp}(w,s)}{\partial w}$ and $\frac{\partial^2 V_{\arp}(w,s)}{\partial w^2}$ exist almost everywhere.
  Moreover, $\frac{\partial V_{\arp}(w,s)}{\partial w}=\frac{\partial C*(\ts_s(\arp)+w)}{\partial \down \xi}$. 
  By part 3 of Lemma~\ref{prop:mono-boundary-con}, $\frac{\partial V_{\arp}(w,s)}{\partial w}$ monotone in $s$.
  So the partial derivatives $\frac{\partial^2 V_{\arp}(w,s)}{\partial s \partial w}$ exists  almost everywhere and hence $\frac{\partial V_{\arp}(w,s)}{\partial s}$ exists almost everywhere.

Then by the optimality condition we have
\begin{equation*}
    V_{\arp}(w,s)
    \le \E\left[\int_0^{\spt'} e^{-\gamma t}h(\arp(s+t)+w+W_t) dt\right]
    +\E\big[e^{-\gamma {\spt'}}V_{\arp}(w+W_{\spt'},s+\spt')\big].
  \end{equation*} 
for any stopping time $\tau'$. 
Combining with the existence of above three the partial derivatives, we immediately derive that  \begin{equation*}
    \frac{\partial V_{\arp}(w,s)}{\partial s}
    +\frac{\sigma^2}{2}\frac{\partial^2 V_{\arp}(w,s)}{\partial w^2}
    +\mu\frac{\partial V_{\arp}(w,s)}{\partial w}-\gamma V_{\arp}(w,s)
    +h(\arp(s)+w) \ge 0
  \end{equation*}
holds for almost every $(w,s) \in \R \times \R_+$.
\end{proof}

\begin{proof}[Proof of Proposition~\ref{prop:characterization of the mapping}]
We only prove the result for $\up \psi(\arp,\down Y,\omega)$.
Suppose the above equation does not hold, i.e., there exists $t$ such that $\frac{\partial C(\arp_t,0,0)}{\partial\up \xi}>0$ and $\up \psi$ increases at $t$.

If $\up \psi(t) > \up \psi(t-)$, then there must exist $\epsilon,\delta>0$ such that $\up \psi(t)-\up \psi(t-)>\epsilon$ and, for any $\arp' \in \mathbb D$ that satisfies $\dist_{\gamma}(\arp',\arp_t)<\epsilon+\delta
\frac{\arp_t(\lt)}{\gamma}$ and $\frac{\partial C(\arp',0,0)}{\partial\up \xi}>0$. 
Hence, the following upward adjustment
\begin{equation*}
  Y^{\uparrow'}(u)=\left\{ 
 \begin{array} {ll}
  \up \psi(u)-\epsilon, & u \in [t,t+\delta),\\
  \up \psi(u), & \textrm{ otherwise}
  \end{array}
  \right.
\end{equation*} 
is strictly less than $\up \psi$.
Following a similar argument as in the proof of Proposition~\ref{prop:order and continuous on one-side}, we can show that $Y^{\uparrow'} \in \up \Pi(\arp,Y^{\downarrow},\omega)$, which implies that $\up \psi$ cannot be the infimum, a contradiction.

If $\up \psi(t)=\up \psi(t-)$, there must exist $\epsilon,\delta>0$ such that $\up \psi(s)-\up \psi(s-)>\epsilon$ and $\frac{\partial C(\arp_s,0,0)}{\partial\up \xi}>0$ for $t\leq s \leq t+\delta$.
Then, the following upward adjustment
\begin{equation*}
  Y^{\uparrow'}(u)=\left\{ 
 \begin{array} {ll}
  \up \psi(t), & u \in [t,t+\delta),\\
  \up \psi(u), & \textrm{ otherwise}
  \end{array}
  \right.
\end{equation*} 
is strictly less than $\up \psi$.
Similarly, we can show that $Y^{\uparrow'} \in \up \Pi(\arp,Y^{\downarrow},\omega)$ which implies that $\up \psi$ cannot be the infimum, again a contradiction. 
Thus, the proposition holds.
\end{proof}

\begin{proof}[Proof of Lemma~\ref{lemma:close-2}]

The proof is quite complicated, thus we give a road map. 
Essentially, we prove that, for any fixed $T>0$,
\begin{eqnarray}
  && C^{\delta}(\arp,\pi^*)
     -C^*(\arp^{\delta}_0)+\E\left[e^{-\gamma T}C^{*}(\arp^{\delta}_T)\right]-(2\E N(T)+2)M\delta \nonumber \\
  & &+\E\left[\int_0^{T} e^{-\gamma t}
      \frac{\partial C(\arp_t^{\delta},0,0)}{\partial \up \xi} d Y^{\uparrow*}(t)
      +\int_0^{T}e^{-\gamma t} 
      \frac{\partial C(\arp_t^{\delta},0,0)}{\partial \down \xi} d Y^{\downarrow*}(t)\right]\nonumber \\
  &\leq& \E\left[
         \int_{T}^{\infty} e^{-\gamma t}h(\arp^{\delta}_t(0))dt
         + \up k \int_T^{\infty} e^{-\gamma t} d Y^{\uparrow*}(t)
         + \down k\int_T^{\infty}e^{-\gamma t} d Y^{\downarrow*}(t)
         \right].\label{eq:delta-change-cost-T}
\end{eqnarray}
Once this is proven, let $R_1(\arp,\delta,T)=\E\left[\int_0^{T} e^{-\gamma t}
    \frac{\partial C(\arp_t^{\delta},0,0)}{\partial \up \xi} d Y^{\uparrow*}(t)
   +\int_0^{T}e^{-\gamma t} 
   \frac{\partial C(\arp_t^{\delta},0,0)}{\partial \down \xi} d Y^{\downarrow*}(t)\right]$
and 
    $R_2(T)=\E\left[
    \int_{T}^{\infty} e^{-\gamma t}h(\arp^{\delta}_t(0))dt
    + \up k \int_T^{\infty} e^{-\gamma t} d Y^{\uparrow*}(t)
    + \down k\int_T^{\infty}e^{-\gamma t} d Y^{\downarrow*}(t)
  \right]-\E[e^{-\gamma T}C^*(\arp^{\delta}_T)]$.
Then, \eqref{eq:delta-change-cost-T} becomes
\begin{equation}
\begin{split}
 C^{\delta}(\arp,\pi^*) \leq C^*(\arp^\delta_0)+(2\E N(T)+2)M\delta -R_1(\arp,\delta,T)+R_2(T).
\end{split}
\label{eq:final-delta}
\end{equation}
By \eqref{equ:feasiblecondition} and the Lipschitz continuity of $C^*(\arp)$, we immediately get that $R_2(T) \to 0$ as $T \to \infty$.
For $R_1(\arp,\delta,T)$, it is easy to see that $\arp_t^{\delta} \to \arp_t$ as $\delta \to 0$, so $\frac{\partial C(\arp_t^{\delta},0,0)}{\partial \up \xi}$ converges to $\frac{\partial C(\arp_t,0,0)}{\partial \up \xi}$ by part~1 of Lemma~\ref{prop:mono-boundary-con}.
By the Lebesgue's Dominated Convergence Theorem, the upward adjustment cost
$\E\left[ \int_0^{T}e^{-\gamma t}\frac{\partial C(\arp_t^{\delta},0,0)}{\partial \up \xi}dY^{\uparrow*}(t)\right]$ converges to 
$\E\left[\int_0^{T} e^{-\gamma t} \frac{\partial C(\arp_t,0,0)}{\partial \up \xi} d Y^{\uparrow*}(t) \right]$, which equals to $0$ by Proposition~\ref{prop:characterization of the mapping}.
Similarly, for the downward adjustment cost, we have $\E\left[ \int_0^{T}e^{-\gamma t}\frac{\partial C(\arp_t^{\delta},0,0)}{\partial \down \xi}dY^{\downarrow*}(t)\right]$ converges to 0. 
Thus, $R_1(\arp,\delta,T) \to 0$ as $\delta\to 0$. 
Finally, since $|C^*(\arp)-C^*(\arp_0^{\delta})|\leq M\delta$, the lemma holds.

The remaining of this proof is devote to showing \eqref{eq:delta-change-cost-T}.
To this end, we apply the following double telescoping to  
$C^*(\arp_0^{\delta})-\E\left[ e^{-\gamma T}C^*(\arp^{\delta}_T)\right]$ in order to approximate $C^{\delta}(\arp,\pi^*)$.

\begin{enumerate}
\item 
In the first telescoping, we write $C^*(\arp_0^{\delta})-\E\left[ e^{-\gamma T}C^*(\arp^{\delta}_T)\right]$ according to the partition of the interval $[0,T]$ by  $0=\tau^{\delta}_{0}<\tau^{\delta}_{1}<\ldots<\tau^{\delta}_{N(T)} \le T$. 
\begin{eqnarray}
&& C^*(\arp_0^{\delta})-\E\left[ e^{-\gamma T}C^*(\arp^{\delta}_T)\right]
     \nonumber\\
&=& \E\sum_{k=1}^{N(T)} 
     \left[
     e^{-\gamma {\tau^{\delta}_{k-1}}}C^{*}\left(\arp^{\delta}_{\tau^{\delta}_{k-1}}\right)
     -e^{-\gamma {\tau^{\delta}_{k}}}C^{*}\left(\arp^{\delta}_{\tau^{\delta}_{k}}\right)
     \right]
     + \E\left[
     e^{-\gamma {\tau^{\delta}_{N(T)}}}C^{*}\left(\arp^{\delta}_{\tau^{\delta}_{N(T)}}\right)
     -e^{-\gamma T}C^{*}\left(\arp^{\delta}_{T}\right)
     \right]\nonumber\\ 
&=&\E\sum_{k=1}^{N(T)} \left[
   \E\left[ 
   e^{-\gamma {\tau^{\delta}_{k-1}}}C^{*}\left(\arp^{\delta}_{\tau^{\delta}_{k-1}}\right)
   -e^{-\gamma \tau^{\delta}_{k}}C^{*}\left(\arp^{\delta}_{\tau^{\delta}_{k}}\right)
   \left. \right|\mathscr{F}_{\tau^{\delta}_{k-1}}\right]
   \right] \label{eq:telescope-1}\\
  && +\E\left[
     \E\left[ \left.
     e^{-\gamma {\tau^{\delta}_{N(T)}}}C^{*}\left(\arp^{\delta}_{\tau^{\delta}_{N(T)}}\right)
    -e^{-\gamma T}C^{*}\left(\arp^{\delta}_{T}\right)
    \right| \mathscr{F}_{\tau^{\delta}_{N(T)}}\right] 
    \right]. \label{eq:telescope-tail}
\end{eqnarray}

\item
Next, we examine all the terms in \eqref{eq:telescope-1} and \eqref{eq:telescope-tail} and apply a sub-telescoping on each of them. 
We construct a partition of the interval $\left[\tau^{\delta}_{k-1},\tau^{\delta}_{k}\right]$ by  $\tau^{\delta}_{k-1}=\iota_{k,0}<\iota_{k,1}<\ldots<\iota_{k,j_k}=\tau^{\delta}_{k}$ for any $0<\epsilon<\delta$ where
\begin{eqnarray*}
   \hat{\iota}_{k,j} &=&
   \inf\left\{u:u > \iota_{k,j-1},
 (Y^{\uparrow*}(u)-Y^{\uparrow*}(\iota_{k,j-1})) \vee(Y^{\downarrow*}(u)-Y^{\downarrow*}(\iota_{k,j-1}))
   \geq \frac{\epsilon}{2} \right\},\\
  \iota_{k,j} &=& \hat{\iota}_{k,j} \wedge  (\iota_{k,j-1}+\epsilon)\wedge \tau^{\delta}_{k+1}, 
\end{eqnarray*} 
for $j=1,2,\cdots,j_k$. 
It's obvious that $j_k$ is almost surely finite. 
We define $Y^{\uparrow \epsilon}$ and $Y^{\downarrow \epsilon}$ piece-wisely on the interval $\left[\tau^{\delta}_{k-1},\tau^{\delta}_k \right]$ as
$Y^{\uparrow \epsilon}(u)=Y^{\uparrow*}(\iota_{k,j})$ and $Y^{\downarrow \epsilon}(u)=Y^{\downarrow*}(\iota_{k,j})$ for $\iota_{k,j} \leq u < \iota_{k,j+1}$.
It is obvious that they are step functions with jump sizes bounded by $\frac{\epsilon}{2}$. 
Let $\arp_t^{\epsilon}$ be the state at time $t$ under policy $(Y^{\uparrow\epsilon},Y^{\downarrow\epsilon})$ with the initial profile $\arp$ and define
\begin{eqnarray*}
  \arp^{\delta,\epsilon}_t = \left\{
  \begin{array}{ll}
    \arp^{\epsilon}_t-\delta, & \textrm{ if }\ t< \tau^{\delta}_1,\\
    \arp^{\epsilon}_t+\delta, & \textrm{ if }\ \tau^{\delta}_{2j-1}\leq t <\tau^{\delta}_{2j},\\
    \arp^{\epsilon}_t-\delta, & \textrm{ if }\ \tau^{\delta}_{2j}\leq t <\tau^{\delta}_{2j+1}.
  \end{array}
  \right.
\end{eqnarray*}
\end{enumerate}

For $k=1,2,\cdots,N(t)$, based on the second step of telescoping, we estimate \eqref{eq:telescope-1} as 
\begin{align}
 & \E\left[\E \left[ \left.
   e^{-\gamma \tau^{\delta}_{k-1}}C^{*}\left(\arp^{\delta}_{\tau^{\delta}_{k-1}}\right)
  -e^{-\gamma \tau^{\delta}_{k}}C^{*}\left(\arp^{\delta}_{\tau^{\delta}_{k}}\right)
    \right|\mathscr{F}_{\tau^{\delta}_{k-1}}\right]\right] \nonumber\\
=& \E\left[e^{-\gamma \tau^{\delta}_{k-1}}
   \left[ C^{*}\left(\arp^{\delta}_{\tau^{\delta}_{k-1}}\right)-C^{*}\left(\arp^{\delta,\epsilon}_{\tau^{\delta}_{k-1}}\right)\right]\right]
    -\E\left[ \left.
   e^{-\gamma \tau^{\delta}_{k}}         
\left[C^{*}\left(\arp^{\delta}_{\tau^{\delta}_{k}}\right)-C^{*}\left(\arp^{\delta,\epsilon}_{\tau^{\delta}_{k}}\right)\right]
   \right|\mathscr{F}_{\tau^{\delta}_{k-1}}\right] \nonumber\\
 & +\E\left[\E\left[ \left.
   e^{-\gamma \tau^{\delta}_{k-1}}C^{*}\left(\arp^{\delta,\epsilon}_{\tau^{\delta}_{k-1}}\right)
  -e^{-\gamma \tau^{\delta}_{k}}C^{*}\left(\arp^{\delta,\epsilon}_{\tau^{\delta}_{k}}\right)
  \right|\mathscr{F}_{\tau^{\delta}_{k-1}}\right]\right] \nonumber\\
\geq& -2M\epsilon +\E\left[\E\left[ \left.
   e^{-\gamma \tau^{\delta}_{k-1}}C^{*}\left(\arp^{\delta,\epsilon}_{\tau^{\delta}_{k-1}}\right)
  -e^{-\gamma \tau^{\delta}_{k}}C^{*}\left(\arp^{\delta,\epsilon}_{\tau^{\delta}_{k}}\right)
  \right|\mathscr{F}_{\tau^{\delta}_{k-1}}\right]\right] \nonumber\\
\begin{split} 
= &-2M \epsilon + \E\left[\E\left[ \left.
  \sum_{j=1}^{j_k}
  e^{-\gamma \iota_{k,j-1}}C^{*}\left(\arp^{\delta,\epsilon}_{\iota_{k,j-1}}\right)
 -e^{-\gamma \iota_{k,j}}C^{*}\left(\arp^{\delta,\epsilon}_{\iota_{k,j}-}\right)
\right|\mathscr{F}_{\tau^{\delta}_{k-1}}\right]\right] \\
  &+\E \left[ \sum_{j=1}^{j_k}
  e^{-\gamma \iota_{k,j}}C^{*}\left(\arp^{\delta,\epsilon}_{\iota_{k,j}}\right)
 -e^{-\gamma \iota_{k,j}}C^{*}\left(\arp^{\delta,\epsilon}_{\iota_{k,j}-}\right)
   \right].
\end{split}\label{eq:telescope-2}
\end{align} 
The last equality follows as a result of telescoping on the partition $\tau^{\delta}_{k-1}=\iota_{k,0}<\iota_{k,1}<\ldots<\iota_{k,j_k}=\tau^{\delta}_{k}$. 
Since there is no upward or downward adjustment during $[\iota_{k,j-1},\iota_{k,j})$, the second term in \eqref{eq:telescope-2} becomes 
\begin{align*}
&  \E\left[\E\left[ \left.\sum_{j=1}^{j_k}
     e^{-\gamma \iota_{k,j-1}} 
    \E\left[ \left.
     C^{*}\left(\arp^{\delta,\epsilon}_{\iota_{k,j-1}}\right)
    -e^{-\gamma (\iota_{k,j}-\iota_{k,j-1})}C^{*}\left(\arp^{\delta,\epsilon}_{\iota_{k,j}-}\right)
    \right|\mathscr{F}_{\iota_{k,j-1}}\right]
    \right|\mathscr{F}_{\tau^{\delta}_{k-1}}\right]\right]\\
=&\E\left[\E\left[ \left. \sum_{j=1}^{j_k}
     e^{-\gamma \iota_{k,j-1}} 
    \E\left[ \left.
    V_{\arp^{\delta\epsilon}_{\iota_{k,j-1}}}(0,0)-e^{-\gamma (\iota_{k,j}-\iota_{k,j-1})}
   V_{\arp^{\delta\epsilon}_{\iota_{k,j-1}}}(W_{\iota_{k,j}}-W_{\iota_{k,j-1}},\iota_{k,j}-\iota_{k,j-1})
    \right|\mathscr{F}_{\iota_{k,j-1}}\right]
    \right|\mathscr{F}_{\tau^{\delta}_{k-1}}\right]\right] \\
=& \E\left[\E\left[ \left. \sum_{j=1}^{j_k}
     e^{-\gamma \iota_{k,j-1}}
   \E\left[ \left.
    \int_0^{\iota_{k,j}-\iota_{k,j-1}} 
   e^{-\gamma u}h\left(\arp^{\delta,\epsilon}_{\iota_{k,j-1}}(u)+W_u\right)du
    \right|\mathscr{F}_{\iota_{k,j-1}}\right]
    \right|\mathscr{F}_{\tau^{\delta}_{k-1}}\right]\right]\\
=& \E\left[\E\left[ \left. \sum_{j=1}^{j_k}
    \int_{\iota_{k,j-1}}^{\iota_{k,j}} 
   e^{-\gamma u}h\left(\arp^{\delta,\epsilon}_{u}(0) \right) du
    \right|\mathscr{F}_{\tau^{\delta}_{k-1}}\right]\right]\\
=& \E\left[
    \int_{\tau^{\delta}_{k-1}}^{\tau^{\delta}_{k}} 
   e^{-\gamma u}h\left(\arp^{\delta,\epsilon}_{u}(0)\right) du
    \right]
   \to 
   \E\left[
   \int_{\tau^{\delta}_{k-1}}^{\tau^{\delta}_{k}} e^{-\gamma u}h(\arp^{\delta}_u(0))
    du \right] \quad \mbox{as $\epsilon \to 0$}.  
\end{align*}
By the definition of $\arp^{\delta}_t$ and $\arp^{\delta,\epsilon}_t$, we have $\arp^{\delta,\epsilon}_t\in\Xi$ for any $t \ge 0$, which allows us to apply Theorem~\ref{thm:PDE} and Corollary~\ref{cor:general-ine-pde} to the second equality with $\arp'=\arp^{\delta,\epsilon}_{\iota_{k,j-1}}$ as an initial state. 
Since $\arp^{\delta,\epsilon}_{u}(0) \to \arp^{\delta}_{u}(0)$ and $h\left(\arp^{\delta,\epsilon}_{u}(0)\right)$ is dominated by $h(\arp^{\delta}_{t}(0))+M\delta$
as $\epsilon \to 0$, convergence is established by the Lebesgue's Dominated Convergence Theorem.

Denote $\up \Delta_{k,j}=Y^{\uparrow\epsilon}(\iota_{k,j})-Y^{\uparrow\epsilon}(\iota_{k,j}-)$ and $\down \Delta_{k,j}=Y^{\downarrow\epsilon}(\iota_{k,j})-Y^{\downarrow\epsilon}(\iota_{k,j}-)$. Then, the third term in \eqref{eq:telescope-2} can be written as 
\begin{align}
\begin{split}
& \E\left[ \sum_{j=1}^{j_k}
     e^{-\gamma \iota_{k,j}}C\left(\arp^{\delta,\epsilon}_{\iota_{k,j}-},\up \Delta_{k,j},\down \Delta_{k,j}\right)
    -e^{-\gamma \iota_{k,j}}C\left(\arp^{\delta,\epsilon}_{\iota_{k,j}-},0,0\right)
    \right]\\
& +\E\left[ \sum_{j=1}^{j_k} e^{-\gamma \iota_{k,j}}
    \phi\left(\up \Delta_{k,j},\down \Delta_{k,j}\right)      
     \right]
  +\E\left[e^{-\gamma \iota_{k,j}}
    \left[C^{*}\left(\arp^{\delta,\epsilon}_{\iota_{k,j_k}}\right)
    -C^{*}\left(\arp^{\delta,\epsilon}_{\iota_{k,j_k}}+(-1)^{k}2\delta\right)\right]
    \right]
\end{split} \label{eq:controlcost}\\
\begin{split}
\geq &\E\left[ \sum_{j=1}^{j_k}
     e^{-\gamma \iota_{k,j}}C\left(\arp^{\delta,\epsilon}_{\iota_{k,j}-},\up \Delta_{k,j},\down
\Delta_{k,j}\right)
    -e^{-\gamma \iota_{k,j}}C\left(\arp^{\delta,\epsilon}_{\iota_{k,j}-},0,0\right)
    \right] \\
  &+\E\left[ \sum_{j=1}^{j_k} e^{-\gamma \iota_{k,j}}
    \phi\left(\up \Delta_{k,j},\down \Delta_{k,j}\right)      
     \right]-2M \delta
\end{split}
    \label{eq:telescope-2-tail}
\end{align}
where the last term in \eqref{eq:controlcost} is due to the fact that, in addition to the jumps $(\up \Delta_{k,j_k},\down \Delta_{k,j_k})$, $\arp^{\delta,\epsilon}_t$ also includes the jump caused by $\delta$ at $\iota_{k,j_k}=\tau^{\delta}_k$. 
The second term in \eqref{eq:telescope-2-tail} is the total discounted ordering cost under policy $(Y^{\uparrow\epsilon},Y^{\downarrow\epsilon})$ and will converge to 
$\E\left[\up k  \int_{\tau^{\delta}_{k-1}}^{\tau^{\delta}_{k}} e^{-\gamma
t} d Y^{\uparrow*}(t)\right]
+\E\left[\down k  \int_{\tau^{\delta}_{k-1}}^{\tau^{\delta}_{k}}e^{-\gamma
t} d Y^{\downarrow*}(t)\right]$. 
The first term in \eqref{eq:telescope-2-tail} can be written as follows for some $(u_1(\omega),u_2(\omega)) \in [0,\frac{\epsilon}{2}] \times [0,\frac{\epsilon}{2}]$, which is also a discrete Riemann sum of an integral
\begin{align*}
 &  \E\left[ \sum_{j=1}^{j_k} e^{-\gamma \iota_{k,j}}
   \left(\frac{\partial C\left(\arp^{\delta,\epsilon}_{\iota_{k,j}-},u_1(\omega),u_2(\omega)\right)}{\partial \up\ja} \up \Delta_{k,j}
  +\frac{\partial C\left(\arp^{\delta,\epsilon}_{\iota_{k,j}-},u_1(\omega),u_2(\omega)\right)}{\partial \down \ja} 
   \down \Delta_{k,j}\right)
   \right] \\
& \to
   \E\left[\int_{\tau^{\delta}_{k-1}}^{\tau^{\delta}_{k}} e^{-\gamma t} 
     \frac{\partial C(\arp_t^{\delta},0,0)}{\partial \up \xi} d Y^{\uparrow*}(t)
     +\int_{\tau^{\delta}_{k-1}}^{\tau^{\delta}_{k}} e^{-\gamma t}
     \frac{\partial C(\arp_t^{\delta},0,0)}{\partial \down \xi} d Y^{\downarrow*}(t)\right]
\end{align*}
because $\max\limits_{j=1,2,\cdots,j_k} \Delta_{k,j} \to 0$ as  $\epsilon \to 0$.
Letting $\epsilon \to 0$, each term in \eqref{eq:telescope-1} is greater than
\begin{equation}
\label{eq:telescope-1-all}
\begin{split}
 & -2M\delta
   +\E\left[\int_{\tau^{\delta}_{k-1}}^{\tau^{\delta}_{k}} e^{-\gamma t} 
     \frac{\partial C(\arp_t^{\delta},0,0)}{\partial \up \xi} 
   d Y^{\uparrow*}(t)
     +\int_{\tau^{\delta}_{k-1}}^{\tau^{\delta}_{k}} e^{-\gamma t}
     \frac{\partial C(\arp_t^{\delta},0,0)}{\partial \down \xi} 
   d Y^{\downarrow*}(t)\right]\\
 & +\E\left[
   \int_{\tau^{\delta}_{k-1}}^{\tau^{\delta}_{k}} e^{-\gamma u}h(\arp^{\delta}_u(0)) du
   \right]
  +\E\left[\up k  \int_{\tau^{\delta}_{k-1}}^{\tau^{\delta}_{k}} e^{-\gamma t} 
    d Y^{\uparrow*}(t)\right]
  +\E\left[\down k  \int_{\tau^{\delta}_{k-1}}^{\tau^{\delta}_{k}}e^{-\gamma t} 
   d Y^{\downarrow*}(t)\right].
\end{split}
\end{equation}
Following the same argument, \eqref{eq:telescope-tail} is greater than
\begin{equation}
\label{eq:telescope-1-all-tail}
\begin{split}
 & -2M\delta
   +\E\left[\int_{\tau^{\delta}_{N(T)}}^{T} e^{-\gamma t} 
     \frac{\partial C(\arp_t^{\delta},0,0)}{\partial \up \xi} 
     d Y^{\uparrow*}(t)
     +\int_{\tau^{\delta}_{N(T)}}^{T} e^{-\gamma t}
     \frac{\partial C(\arp_t^{\delta},0,0)}{\partial \down \xi} d Y^{\downarrow*}(t)\right]\\
 & +\E\left[
   \int_{\tau^{\delta}_{N(T)}}^{T} e^{-\gamma u}h(\arp^{\delta}_u(0)) du
   \right]
  +\E\left[\up k  \int_{\tau^{\delta}_{N(T)}}^{T} e^{-\gamma t} 
    d Y^{\uparrow*}(t)\right]
  +\E\left[\down k  \int_{\tau^{\delta}_{N(T)}}^{T}e^{-\gamma t} 
   d Y^{\downarrow*}(t)\right].
\end{split}
\end{equation}
Plugging \eqref{eq:telescope-1-all} and \eqref{eq:telescope-1-all-tail} into \eqref{eq:telescope-1} and \eqref{eq:telescope-tail}, we have that
\begin{align*} 
& C^*(\arp^{\delta}_0)-\E\left[e^{-\gamma T}C^{*}(\arp^{\delta}_T)\right]-\E[2N(T)+2]M\delta\\
\geq& \E\left[
   \int_{0}^{T} e^{-\gamma u}h(\arp^{\delta}_u(0)) du
   \right]
   +\up k \E\left[\int_0^{T} e^{-\gamma t} d Y^{\uparrow*}(t)\right]
   +\down k \E\left[\int_0^{T}e^{-\gamma t} d Y^{\downarrow*}(t)\right]\\
  &+\E\left[\int_0^{T} e^{-\gamma t} 
    \frac{\partial C(\arp_t^{\delta},0,0)}{\partial \up \xi} d Y^{\uparrow*}(t)
   +\int_0^{T} e^{-\gamma t}
    \frac{\partial C(\arp_t^{\delta},0,0)}{\partial \down \xi} d Y^{\downarrow*}(t)\right].
\end{align*}
Combining the above with the cost function $C^{\delta}(\arp,\pi^*)$ defined in \eqref{eq:C-mod}, we have  \eqref{eq:delta-change-cost-T}.
\end{proof}

\begin{proof}[Proof of Proposition~\ref{prop:costequivalence}]
It follows as
\begin{eqnarray*}
   && \E \left[\int_{\down\lt}^{\infty}e^{-\gamma t} h(H_t)dt \right]=\E\left[\int_{0}^{\infty}e^{-\gamma (t+\down\lt)} h(H_{t+\down\lt}) dt\right]\\
    &=& \E\left[\int_{0}^{\infty}e^{-\gamma (t+\down\lt)} h(H_{0}+W_{t+\down\lt}+\up{\arp}_{0}(t+\down\lt)-\down{\arp}_{0}(t+\down\lt)+\up{Y}(t+\down\lt-\up\lt)-\down{Y}(t)) dt\right]\\
    &=& \E\left[\int_{0}^{\infty}e^{-\gamma (t+\down\lt)} h(W_{t+\down\lt}-W_{t}+W_{t}+H_{0}+\up{\arp}_{0}(t+\down\lt)-\down{\arp}_{0}(\down\lt)+\up{Y}(t-\lt)-\down{Y}(t)) dt\right]\\
    &=& \E\left[\int_{0}^{\infty}e^{-\gamma (t+\down\lt)} h(W_{t+\down\lt}-W_{t}+\arp_0(t)+\up Y(t-\lt)-\down Y (t))dt\right]\\
    &=& \E\left[\int_{0}^{\infty}e^{-\gamma (t+\down\lt)}
     \E\left[\left. h(W_{t+\down\lt}-W_{t}+\arp_t(0)) \right|\arp_t(0) dt
     \right] \right]\\
    &=& \E\left[\int_0^{\infty}e^{-\gamma t} 
   \E\left[e^{-\gamma \down{\lt}}h(\arp_t(0)+\mathcal{N}_{\down\lt})dt
   \right]\right]\\
    &=& \E\left[\int_{0}^{\infty}e^{-\gamma t} \tilde{h}(\arp_t(0)) dt\right].
  \end{eqnarray*}
The cost difference $\E \left[\int_0^{\down{\lt}}e^{-\gamma t}{h}(H_t)
dt \right]$ is a constant because, for $t \leq \down\lt$,
  \begin{align*}
    H_t =H_0 + W_t + \up{\arp}_0(t) - \down{\arp}_0(t) + \up Y(t-\down{\lt}) + \down Y(t-\up{\lt})
    =H_0 + W_t + \up{\arp}_0(t) - \down{\arp}_0(t).
  \end{align*}
\end{proof}
\end{document}